\documentclass[a4paper]{amsart}
\usepackage{amssymb}
\usepackage{verbatim}
\usepackage[UglyObsolete]{diagrams}
\usepackage[dvips]{color}
\usepackage[dvips]{graphicx}
\usepackage{amscd}
\usepackage{mathabx}
\usepackage{arydshln}
\usepackage[color,final]{showkeys}
\definecolor{refkey}{rgb}{1,0,0}
\definecolor{labelkey}{rgb}{0,0,1}
\usepackage{lscape}
\allowdisplaybreaks
\theoremstyle{plain}
  \newtheorem{thm}{Theorem}[section]
  
  \newtheorem{cor}[thm]{Corollary}
  \newtheorem{prop}[thm]{Proposition}
  \newtheorem{conj}[thm]{Conjecture}

  \newtheorem*{obs*}{Observation}
\theoremstyle{definition}

\theoremstyle{remark}
  \newtheorem{rem}[thm]{Remark}
  
  \newtheorem*{ack}{Acknowledgments}
\newcommand{\Z}{\mathbb{Z}}
\newcommand{\C}{\mathbb{C}}

\newcommand{\Vol}{\operatorname{Vol}}
\newcommand{\CS}{\operatorname{CS}}

\newcommand{\Int}{\operatorname{Int}}

\newcommand{\SL}{\rm{SL}}
\renewcommand{\sl}{\mathfrak{sl}}

\newcommand{\Tor}{\operatorname{Tor}}
\newcommand{\Ad}[1]{\operatorname{Ad}_{#1}}
\newcommand{\ad}[1]{\operatorname{ad}_{#1}}

\newcommand{\diag}{\operatorname{diag}}

\newcommand{\pattern}{D}

\renewcommand{\Im}{\operatorname{Im}}

\newcommand{\Ker}{\operatorname{Ker}}
\numberwithin{equation}{section}
\allowdisplaybreaks
\begin{document}
\title{The twisted Reidemeister torsion of an iterated torus knot}
\author{Hitoshi Murakami}
\address{
Graduate School of Information Sciences,
Tohoku University,
Aramaki-aza-Aoba 6-3-09, Aoba-ku,
Sendai 980-8579, Japan
}
\email{starshea@tky3.3web.ne.jp}
\date{\today}
\begin{abstract}
We calculate the twisted Reidemeister torsion of the complement of an iterated torus knot associated with a representation of its fundamental group to the complex special linear group of degree two.
We also show that the twisted Reidemeister torsions associated with various representations appear in the asymptotic expansion of the colored Jones polynomial.
\end{abstract}
\dedicatory{Dedicated to Professor~Taizo~Kanenobu, Professor~Yasutaka~Nakanishi, and Professor~Makoto~Sakuma for their 60th birthdays}
\thanks{This work was supported by JSPS KAKENHI Grant Number 26400079}
\maketitle
\section{Introduction}
Let $J_N(K;q)$ be the colored Jones polynomial of a knot $K$ in the three-sphere $S^3$ associated with the $N$-dimensional irreducible representation of the Lie algebra $\sl(2;\C)$ \cite{Jones:BULAM385,Kirillov/Reshetikhin:1989}.
We normalize $J_N(K;q)$ so that it is $1$ when $K$ is the unknot.
We also let $\Delta(K;t)$ be the Alexander polynomial $K$ that is normalized so that $\Delta(K;1)=1$ and $\Delta(K;t^{-1})=\Delta(K;t)$ \cite{Alexander:TRAAM28}.
\par
The following conjecture would give a topological interpretation of the colored Jones polynomial.
\begin{conj}[Volume Conjecture \cite{Kashaev:MODPLA95,Kashaev:LETMP97,Murakami/Murakami:ACTAM12001}]
Let $K$ be a knot in $S^3$.
Then we have
\begin{equation*}
  \lim_{N\to\infty}\frac{\log\left|J_N(K;\exp(2\pi\sqrt{-1}/N))\right|}{N}
  =
  \frac{\Vol(S^3\setminus{K})}{2\pi},
\end{equation*}
where $\Vol$ means the simplicial volume \rm{(}or the Gromov norm \cite{Gromov:INSHE82}\rm{)}.
Note that it coincides with the sum of hyperbolic volumes of the hyperbolic pieces in the Jacd--Shalen--Johannson decomposition \rm{(}\cite{Jaco/Shalen:MEMAM79,Johannson:1979}\rm{)} of the knot complement up to constant multiplication.
\end{conj}
\par
We can complexify the conjecture as follows.
\begin{conj}[Complexification of the volume conjecture \cite{Murakami/Murakami/Okamoto/Takata/Yokota:EXPMA02}]
\label{conj:complex_VC}
Let $K$ be a hyperbolic knot and let $\CS(S^3\setminus{K})$ be the $SO(3)$ Chern--Simons invariant associated with the holonomy representation of $\pi_1(S^3\setminus{K})\to\SL(2;\C)$.
Then we have
\begin{equation*}
  \lim_{N\to\infty}\frac{\log J_N(K;\exp(2\pi\sqrt{-1}/N))}{N}
  =
  \frac{\Vol(S^3\setminus{K})+\sqrt{-1}\CS(S^3\setminus{K})}{2\pi}.
\end{equation*}
\end{conj}
Here a knot $K$ is called hyperbolic if its complement $S^3\setminus{K}$ has a complete hyperbolic structure with finite volume.
Note that this structure is induced from the holonomy representation.
\par
We can generalize the complexified volume conjecture as follows (\cite{Gukov:COMMP2005,Gukov/Murakami:FIC2008,Dimofte/Garoufalidis:GEOTO2013}).
\begin{conj}\label{conj:refined_VC}
If  $K$ is hyperbolic, then the following asymptotic equivalence holds:
\begin{equation*}
\begin{split}
  J_N(K;\exp(2\pi\sqrt{-1}/N))
  \underset{N\to\infty}{\sim}&
  2\pi^{3/2}\Tor(K)^{-1/2}\left(\frac{N}{2\pi\sqrt{-1}}\right)^{3/2}
  \\
  &\times
  \exp\left(\frac{N}{2\pi}\left(\Vol(S^3\setminus{K})+\sqrt{-1}\CS(S^3\setminus{K})\right)\right),
\end{split}
\end{equation*}
where $\Tor(K)$ is the homological Reidemeister torsion of $S^3\setminus\Int{N(K)}$ twisted by the adjoint action of the holonomy representation associated with the meridian.
Here $f(N)\underset{N\to\infty}{\sim}g(N)$ means that $f(N)=\bigl(1+o(1)\bigr)g(N)$ for $N\to\infty$, and $N(k)$ is the regular neighborhood of $K$ in $S^3$ and $\Int$ denotes the interior.
\end{conj}
This conjecture is proved for the figure-eight knot by J.~Andersen and S.~Hansen \cite{Andersen/Hansen:JKNOT2006} (see also \cite[Theorem~1.3]{Murakami:JTOP2013}), and for the $5_2$ knot by T.~Ohtsuki \cite{Ohtsuki:ki52}.
See also \cite{Ohtsuki/Takata:GEOTO2015} for related topics.
\par
In \cite{Gukov/Murakami:FIC2008} the following conjecture was proposed for a further generalization:
\begin{conj}\label{conj:parameter_VC}
Suppose that $K$ is a hyperbolic knot.
Then we have
\begin{multline*}
  J_N\left(K;\exp\bigl(2\pi\sqrt{-1}+u\bigr)/N\right)
  \\
  \underset{N\to\infty}{\sim}
  \frac{\sqrt{-\pi}}{2\sinh(u/2)}\Tor(K;u)^{-1/2}
  \left(\frac{N}{2\pi\sqrt{-1}+u}\right)^{1/2}
  \exp\left(\frac{N}{2\pi\sqrt{-1}+u}S(u)\right)
\end{multline*}
if $|u|\ne0$ is small enough, where $\Tor(K;u)$ is the homological Reidemeister torsion of $S^3\setminus\Int{N(K)}$ twisted by the adjoint action of the representation $\rho_{u}$ parametrized by $u$, associated with the meridian $\mu$, and $S(u)$ determines the $\SL(2;\C)$ Chern--Simons invariant of $\rho_{u}$.
Note that $\rho_{0}$ is the holonomy representation.
\end{conj}
This conjecture is proved in \cite{Murakami:JPJGT2007} for the case where $K$ is the figure-eight knot and $u$ is real.
See also \cite{Murakami/Yokota:JREIA2007,Gukov:COMMP2005,Dimofte/Gukov/Lenells/Zagier:CNTP2010,Dimofte/Gukov:CS,Zagier:2010}.
\par
A knot is called simple if every incompressible torus in its complement is boundary parallel.
Any hyperbolic knot is known to be simple.
If a knot is simple and non-hyperbolic, then it is a torus knot.
For coprime positive integers $c$ and $d$, let $T(c,d)$ be the torus knot of type $(c,d)$.
Related to Conjecture~\ref{conj:refined_VC}, R.~Kashaev and O.~Tirkkonen proved the following asymptotic expansion \cite{Kashaev/Tirkkonen:ZAPNS2000}.
\begin{thm}
\begin{equation*}
\begin{split}
  &J_N\bigl(T(c,d);\exp(2\pi\sqrt{-1}/N)\bigr)
  \\
  =&
  \left(
    \frac{\pi^{3/2}}{2cd}
    \left(\frac{N}{2\pi\sqrt{-1}}\right)^{3/2}
    \sum_{k=1}^{cd-1}
    (-1)^{k+1}k^2
    \tau(k)
    \exp\left(S(2\pi\sqrt{-1};k)\frac{N}{2\pi\sqrt{-1}}\right)
  \right.
  \\
  &\qquad
  \left.
    +
    \frac{1}{4}
    \sum_{j=1}^{\infty}
    \frac{a_j}{j!}
    \left(\frac{2\pi\sqrt{-1}}{4cdN}\right)^{j-1}
  \right),
\end{split}
\end{equation*}
where
\begin{align*}
  S(\xi;k)
  &:=
  \frac{-(2k\pi\sqrt{-1}-cd\xi)^2}{4cd},
  \\
  &\intertext{and}
  \\
  \tau(k)
  &:=
  (-1)^{k+1}
  \frac{4\sin(k\pi/c)\sin(k\pi/d)}{\sqrt{cd}},
\end{align*}
and $a_j$ is the $2j$-th derivative of $2z\sinh{z}/\Delta\bigl(T(c,d);e^{2z}\bigr)$ at $z=0$ with $\Delta(K;t)$ the normalized Alexander polynomial of a knot $K$.
\end{thm}
In \cite{Dubois/Kashaev:MATHA2007} J.~Dubois and Kashaev proved that $S(2\pi\sqrt{-1};k)$ is the Chern--Simons invariant and $\tau_k^{-2}$ is the homological twisted Reidemeister torsion both of which are associated with a representation $\rho_{k}:\pi_1(S^3\setminus{T(c,d)})\to\SL(2;\C)$ parametrized by an integer $k$ ($0<k<cd$) (see Theorem~\ref{thm:Hikami/Murakami_intro}).
\par
As for Conjecture~\ref{conj:parameter_VC}, K.~Hikami and the author proved the following theorem in \cite{Hikami/Murakami:Bonn}.
\begin{thm}[\cite{Hikami/Murakami:Bonn}]\label{thm:Hikami/Murakami_intro}
Let $\xi$ be a complex number with non-zero real part and non-negative imaginary part.
Then we have the following equality as $N\to\infty$:
\begin{multline*}
  J_{N}\bigl(T(c,d);\exp(\xi/N)\bigr)
  \\
  =
  \frac{1}{\Delta\bigl(T(c,d);\exp\xi\bigr)}
  +
  \frac{\sqrt{-\pi}}{2\sinh(\xi/2)}\sqrt{\frac{N}{\xi}}
  \sum_{k}\tau(k)\exp\left(\frac{N}{\xi}S(\xi;k)\right)
  +o(1).
\end{multline*}
Moreover $\tau(k)^{-2}$ is the homological Reidemeister torsion of $S^3\setminus\Int{N(T(c,d))}$ twisted by the adjoint action of an irreducible representation $\rho_k\colon\pi_1(S^3\setminus\Int{N(T(c,d))})\to\SL(2;\C)$ associated with the meridian $\mu$, and $S(\xi;k)$ is related to the $\SL(2;\C)$ Chern--Simons invariant of $S^3\setminus\Int{N(T(c,d))}$ associated with $\rho_k$, where $N(K)$ denotes the regular neighborhood of $K$ in $S^3$ and $\Int$ denotes the interior, $\SL(2;\C)$ is the set of two by two complex matrices with determinants one.
\end{thm}
See \cite{Hikami/Murakami:Bonn,Murakami:ACTMV2008} for more details.
\par
For non-negative integers $a$ and $b$, let $T(2,2a+1)^{(2,2b+1)}$ be the $(2,2b+1)$-cable of the torus knot of type $(2,2a+1)$ (see Figure~\ref{fig:iterated_torus_knot}).
\begin{figure}[h!]
  \includegraphics[scale=0.3]{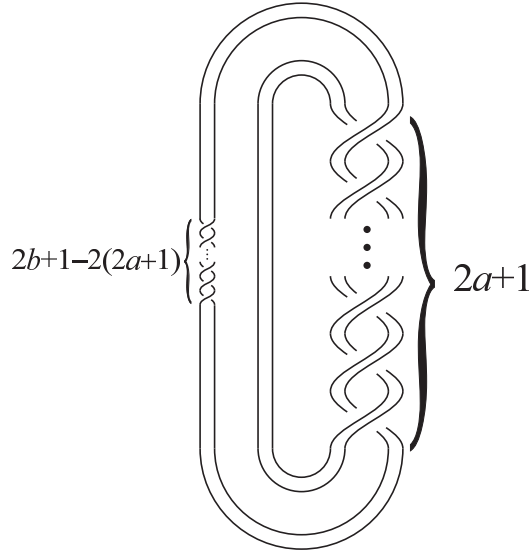}
  \caption{$(2,2b+1)$-cable of the torus knot of type $(2,2a+1)$}
  \label{fig:iterated_torus_knot}
\end{figure}
\begin{rem}
Note that in Figure~\ref{fig:iterated_torus_knot} there are $2b+1-2(2a+1)$ half-twists.
This is because since the torus knot $T(2,2a+1)$ has $2a+1$ crossings, the torus along $T(2,2a+1)$ is given $2(2a+1)$ half-twists.
\end{rem}
We assume that $2b+1>4(2a+1)$.
In the paper \cite{Murakami:ACTMV2014}, the author gave the following equality as $N\to\infty$:
\begin{equation}\label{eq:asymptotic}
\begin{split}
  &J_{N}\left(T(2,2a+1)^{(2,2b+1)};\exp(\xi/N)\right)
  \\
  =&
  \frac{1}{\Delta\left(T(2,2a+1)^{2b+1};\exp\xi\right)}
  \\
  &+
  \frac{\sqrt{-\pi}}{2\sinh(\xi/2)}
  \sqrt{\frac{N}{\xi}}
  \sum_{j}
  \tau_1(\xi;j)
  \exp\left[\frac{N}{\xi}S_1(\xi;j)\right]
  \\
  &+
  \frac{\sqrt{-\pi}}{2\sinh(\xi/2)}
  \sqrt{\frac{N}{\xi}}
  \sum_{k}
  \tau_2(\xi;k)
  \exp\left(\frac{N}{\xi}S_2(\xi;k)\right)
  \\
  &+
  \frac{\pi}{2\sinh(\xi/2)}
  \frac{N}{\xi}
  \sum_{l,m}
  \tau_3(\xi;l,m)
  \exp\left(\frac{N}{\xi}S_3(\xi;l,m)\right)
  +o(1),
\end{split}
\end{equation}
where
\begin{align*}
  \tau_1(\xi;j)
  &:=
  (-1)^j
  \sqrt{\frac{2}{2b+1}}
  \frac{\sin\left(\frac{2(2j+1)\pi}{2b+1}\right)}{\cos\left(\frac{(2j+1)(2a+1)\pi}{2b+1}\right)},
  \\
  S_1(\xi;j)
  &:=
  (2j+1)\xi\pi\sqrt{-1}
  -
  \frac{(2b+1)\xi^2}{2}
  +
  \frac{(2j+1)^2\pi^2}{2(2b+1)},
  \\
  \tau_2(\xi;k)
  &:=
  (-1)^{k+1}
  \sqrt{\frac{2}{2a+1}}
  \frac{\sin\left(\frac{(2k+1)\pi}{2a+1}\right)}{\cosh\left(\frac{(2b+1-4(2a+1))\xi}{2}\right)},
  \\
  S_2(\xi;k)
  &:=
  2(2k+1)\xi\pi\sqrt{-1}
  -
  2(2a+1)\xi^2
  +
  \frac{(2k+1)^2\pi^2}{2(2a+1)},
  \\
  \tau_3(\xi;l,m)
  &:=
  (-1)^{l+m}
  \frac{4}{\sqrt{(2a+1)(2b+1-4(2a+1))}}
  \sin\left(\frac{(2m+1)\pi}{2a+1}\right),
  \\
  S_3(\xi;l,m)
  &:=
  (2l+1)\xi\pi\sqrt{-1}
  -
  \frac{(2b+1)\xi^2}{2}
  \\
  &\hspace{-3mm}+
  \frac{\pi^2\bigl((2l+1)^2(2a+1)+(2m+1)^2(2b+1)-4(2l+1)(2m+1)(2a+1)\bigr)}{2(2a+1)(2b+1-4(2a+1))}.
\end{align*}
See \cite{Murakami:ACTMV2014} for the ranges of the summations.
\par
If we use another normalization $\tilde{J}_{N}(K;q):=J_{N}(K;q)(q^{1/2}-q^{-1/2})$, then we the formula above becomes
\begin{equation}
\begin{split}
  &\tilde{J}_{N}\left(T(2,2a+1)^{(2,2b+1)};\exp(\xi/N)\right)
  \\
  =&
  \tau_0(\xi)
  +
  \sqrt{-\pi}
  \sqrt{\frac{N}{\xi}}
  \sum_{j}
  \tau_1(\xi;j)
  \exp\left(\frac{N}{\xi}S_1(\xi;j)\right)
  \\
  &+
  \sqrt{-\pi}
  \sqrt{\frac{N}{\xi}}
  \sum_{k}
  \tau_2(\xi;k)
  \exp\left(\frac{N}{\xi}S_2(\xi;k)\right)
  \\
  &+
  \pi
  \frac{N}{\xi}
  \sum_{l,m}
  \tau_3(\xi;l,m)
  \exp\left(\frac{N}{\xi}S_3(\xi;l,m)\right)
  +o(1),
\end{split}
\end{equation}
where
\begin{equation*}
  \tau_0(\xi)
  :=
  \frac{2\sinh(\xi/2)}{\Delta(T(2,2a+1)^{2b+1};\exp\xi)}.
\end{equation*}
\par
We studied the $\SL(2;\C)$ Chern--Simons invariants of various representations of $\pi_1(T(2,2a+1)^{(2,2b+1)})$ and concluded that $S_1(\xi;j)$, $S_2(\xi;k)$ and $S_3(\xi;l,m)$ determine these invariants.
We also studied the twisted Reidemeister torsion of those representations and found (non-rigorous) relations to $\tau_1(\xi;j)$, $\tau_2(\xi;k)$ and $\tau_3(\xi;l,m)$.
\par
The purpose of this paper is to prove that $\tau_1(\xi;j)^{-2}$, $\tau_2(\xi;k)^{-2}$ and $\tau_3(\xi;l,m)^{-2}$ are indeed the twisted Reidemeister torsions.
\par
Let $B$ be the solid torus $D^2\times S^1$ and $L$ be the knot in $B$ traveling along the direction of $S^1$ twice with $2b+1$ half-twists as depicted in Figure~\ref{fig:twist_2b+1}.
\begin{figure}[h!]
  \includegraphics[scale=0.3]{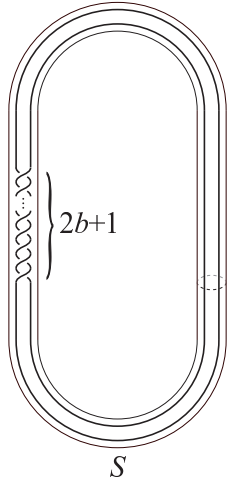}
  \caption{Pattern knot in a solid torus}
  \label{fig:twist_2b+1}
\end{figure}
If $f\colon B\to S^3$ is an embedding so that $f(B)=N(T(2,2a+1))$, then $T(2,2a+1)^{(2;2b+1)}$ is equivalent to the knot $\varphi(L)\subset S^3$.
If we put $E:=S^3\setminus\Int{N(T(2,2a+a)^{(2,2b+1)})}$, $C:=S^3\setminus\Int{N(T(2,2a+1))}$, and $D:=B\setminus\Int{N(L)}$, then $E=C\bigcup_{S}{D}$, where $S=C\cap{D}=\partial{C}=\partial{B}$.
\begin{equation}\label{eq:JSJ}
  \raisebox{-35mm}{\includegraphics[scale=0.3]{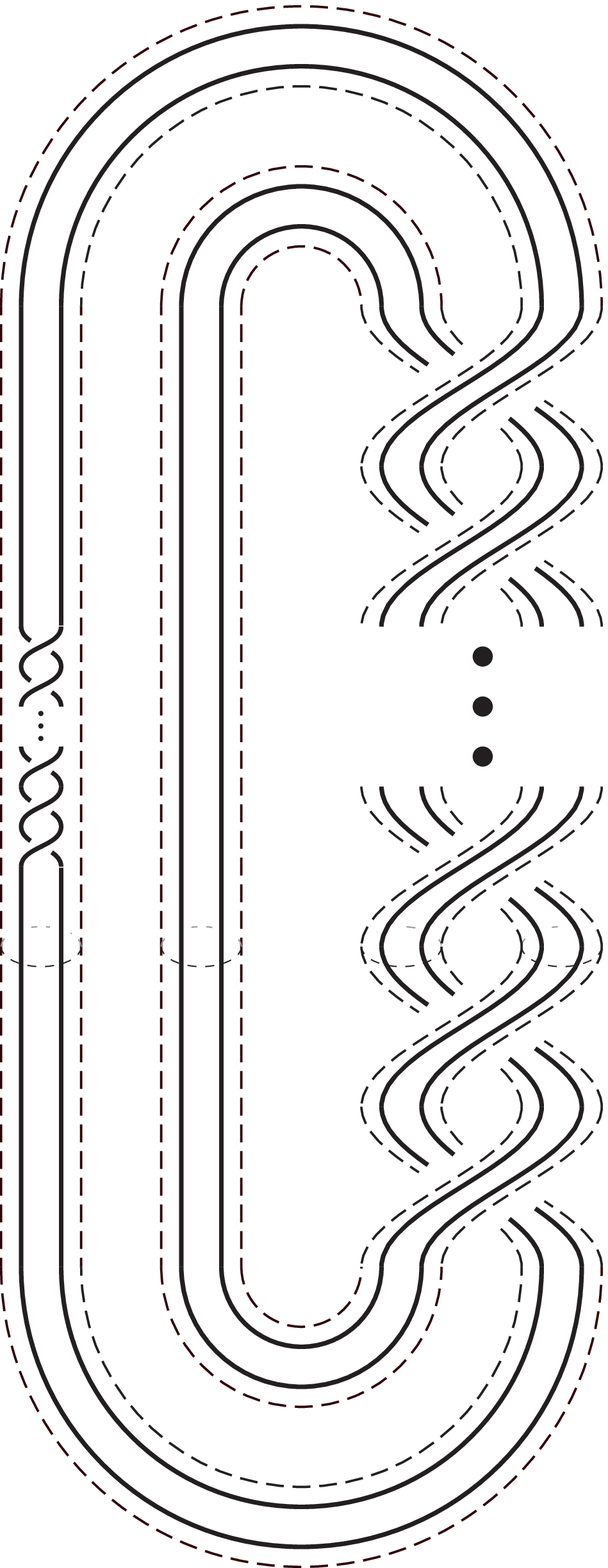}}
  \quad=\quad
  \raisebox{-40mm}{\includegraphics[scale=0.3]{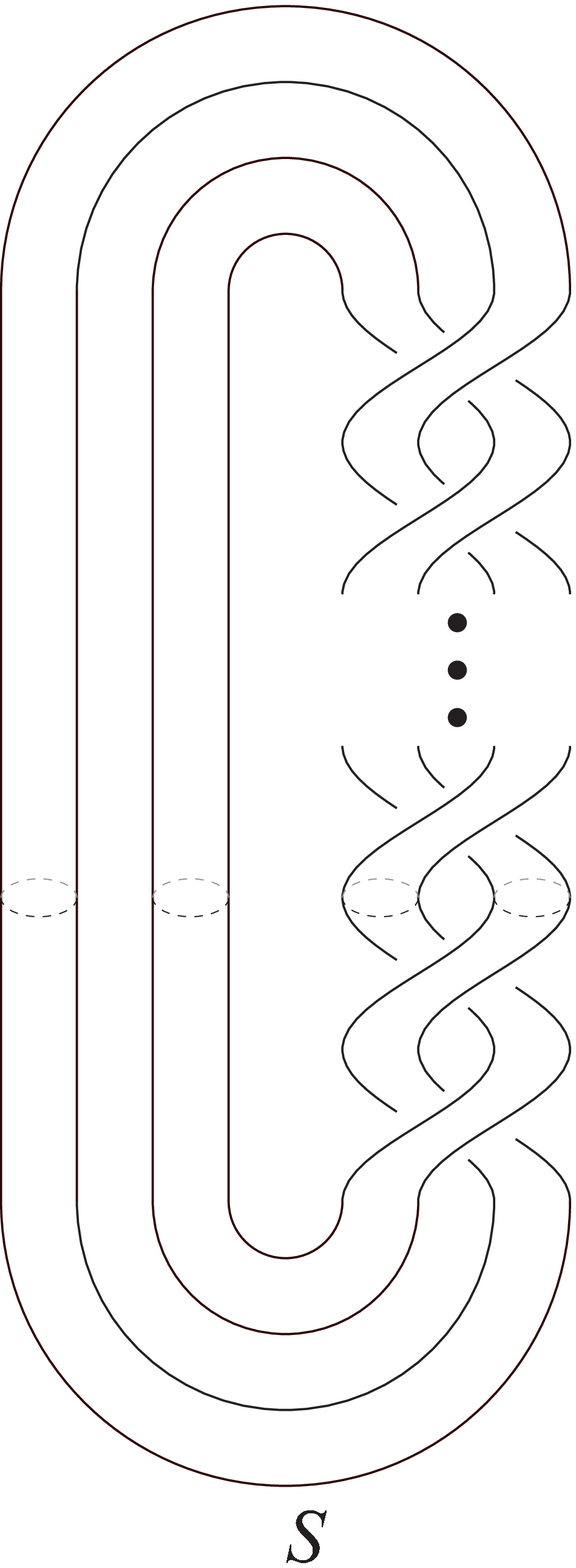}}
  \quad\bigcup\quad
  \raisebox{-33mm}{\includegraphics[scale=0.3]{twist_2b+1.eps}}.
\end{equation}
Put $\Sigma:=\partial{N(T(2,2a+1)^{(2,2b+1)})}$.
\par
Let $\rho^{\rm{AA}}_{\xi}$, $\rho^{\rm{AN}}_{\xi,j}$, $\rho^{\rm{NA}}_{\xi,k}$, and $\rho^{\rm{NN}}_{\xi,l,m}$ be representations of $\pi_1(E,v)$ to $\SL(2;\C)$ defined in \S\ref{sec:rep}, where $v\in{S}$ is the basepoint.
Then we have the following theorem.
\begin{thm}
The quantities $\tau_{0}(\xi)$, $\tau_1(\xi;j)$, $\tau_2(\xi,k)$, and $\tau_3(\xi;l,m)$ given in \eqref{eq:asymptotic} are the homological Reidemeister torsions of $E$ twisted by the adjoint actions of $\rho^{\rm{AA}}_{\xi}$, $\rho^{\rm{AN}}_{\xi,j}$, $\rho^{\rm{NA}}_{\xi,k}$, and $\rho^{\rm{NN}}_{\xi,l,m}$, respectively.
More precisely we have
\begin{align*}
  \tau_{0}(\xi)^{-2}&=\Tor\left(E;\rho^{\rm{AA}}_{\xi}\right),
  \\
  \tau_1(\xi;j)^{-2}&=\Tor\left(E;\rho^{\rm{AN}}_{\xi,j}\right),
  \\
  \tau_2(\xi;k)^{-2}&=\Tor\left(E;\rho^{\rm{NA}}_{\xi,k}\right),
  \\
  \tau_3(\xi;l,m)^{-2}&=\Tor\left(E;\rho^{\rm{NN}}_{\xi,l,m}\right).
\end{align*}
Moreover the torsions above are associated with the following bases of the twisted homologies of $E$:
\begin{itemize}
\item
$\Tor(E;\rho^{\rm{AN}}_{\xi,j})$ is associated with
  \begin{itemize}
  \item
  $\{[\Sigma]\otimes V\}$
  for $H_2(E;\rho^{\rm{AN}}_{\xi,j})=\C$,
  and
  \item
  $\{\mu\otimes V\}$ for $H_1(E;\rho^{\rm{AN}}_{\xi,j})=\C$,
  \end{itemize}
\item
$\Tor(E;\rho^{\rm{NA}}_{\xi,k})$ is associated with
  \begin{itemize}
  \item
  $\{[\Sigma]\otimes W\}$
  for $H_2(E;\rho^{\rm{NA}}_{\xi,k})=\C$,
  and
  \item
  $\{\mu\otimes W\}$ for $H_1(E;\rho^{\rm{NA}}_{\xi,k})=\C$,
  \end{itemize}
\item
$\Tor(E;\rho^{\rm{NN}}_{\xi,l,m})$ is associated with
  \begin{itemize}
  \item
  $\{[S]\otimes\tilde{U},[\Sigma]\otimes\tilde{V}\}$ for $H_2(E;\rho^{\rm{NN}}_{\xi,l,m})=\C^2$,
  and
  \item
  $\{\mu\otimes\tilde{V},\gamma\}$ for $H_1(E;\rho^{\rm{NN}}_{\xi,l,m})=\C^2$.
  \end{itemize}
\end{itemize}
Here
\begin{itemize}
\item
$\mu$ is the meridian of $T(2,2a+1)^{(2,2b+1)}$,
\item
$[S]$ and $[\Sigma]$ are the fundamental class of $H_2(S;\Z)$ and $H_2(\Sigma;\Z)$, respectively,
\item
$V\in\sl(2;\C)$ is invariant under the adjoint action of the image of any element in $\pi_1(\Sigma)$ by $\rho^{\rm{AN}}_{\xi,j}$,
\item
$W\in\sl(2;\C)$ is invariant under the adjoint action of the image of any element in $\pi_1(\Sigma)$ by $\rho^{\rm{NA}}_{\xi,k}$,
\item
$\tilde{U}\in\sl(2;\C)$ is invariant under the adjoint action of the image of any element in $\pi_1(S)$ by $\rho^{\rm{NN}}_{\xi,l,m}$,
\item
$\tilde{V}\in\sl(2;\C)$ is invariant under the adjoint action of the image of any element in $\pi_1(\Sigma)$ by $\rho^{\rm{NN}}_{\xi,l,m}$, and
\item
$\gamma\in H_1(E;\rho^{\rm{NN}}_{\xi,l,m})$ is an element such that $\delta_1(\gamma)=[\tilde{v}\otimes\tilde{U}]\in H_0(S;\rho^{\rm{NN}}_{\xi,l,m})$, where $\delta_1\colon H_1(E;\rho^{\rm{NN}}_{\xi,l,m})\to H_0(S;\rho^{\rm{NN}}_{\xi,l,m})$ is the connecting homomorphism associated with the Mayer--Vietoris exact sequence of the triple $(C,D,S)$ with $S:=C\cap D$, and $\tilde{U}\in\sl(2;\C)$ is invariant under the adjoint action of any element in $\pi_1(S)$.
\end{itemize}
\end{thm}
\begin{ack}
Part of this work was done when the author was visiting the Max Planck Institute for Mathematics and Universit{\'e} Paris Diderot.
The author thanks them for their hospitality.
\par
I owe much of the calculations in this paper to Mathematica \cite{Mathematica}.
\end{ack}
\section{Definition of the Reidemeister torsion}\label{sec:definition}
Let $M$ be a compact manifold and $\pi_1(M)$ be its fundamental group with basepoint $v$.
Let $\rho\colon\pi_1(M)\to\SL(2;\C)$ be a representation.
We will recall the definition of the Reidemeister torsion $\Tor(M;\rho)$ of $M$ associated with $\rho$.
See \cite{Dubois:CANMB2006,Porti:MAMCAU1997} for details.
\par
If we let $\tilde{M}$ be the universal cover of $M$, then $\pi_1(M)$ acts on $\tilde{M}$ as the deck transformation.
The adjoint action $\Ad{\rho(g)}$ of $\rho(g)$ ($g\in\pi_1(M)$ on $\sl(2;\C)$ is defined by $\Ad{\rho(g)}(x):=\rho(g)^{-1}x\rho(g)\in\sl(2;\C)$ ($x\in\sl(2;\C)$).
Put $C_i(M;\rho):=C_i(\tilde{M})\otimes_{\Z[\pi_1(M)]}\sl(2;\C)$, where $\pi_1(M)$ acts on $C_i(\tilde{M})$ by the covering transformation and on $\sl(2l\C)$ by the adjoint action.
We denote by $H_i(M;\rho)$ the homology of the chain complex $\cdots\to C_{i+1}(M;\rho)\xrightarrow{\partial_{i+1}}C_i(M;\rho)\xrightarrow{\partial_i}C_{i-1}(M;\rho)\to\cdots.$
\par
Put
\begin{itemize}
\item $C_{i}:=C_{i}(M;\rho)$,
\item $B_{i}:=\Im\partial_{i+1}$,
\item $Z_{i}:=\Ker\partial_{i}$,
\item $H_{i}:=H_{i}(M;\rho)$.
\end{itemize}
Note that $H_{i}=Z_{i}/B_{i}$.
Let
\begin{itemize}
\item
$\mathbf{c}_{i}:=\{c_{i,1},c_{i,2},\dots,c_{i,l_i}\}$ be a basis of $C_{i}$,
\item
$\mathbf{b}_{i}:=\{b_{i,1},b_{i,2},\dots,b_{i,m_i}\}$ be a set of vectors such that $\{\partial_{i}(b_{i,1}),\partial_{i}(b_{i,2}),\dots,\partial_{i}(b_{i,m_{i}})\}$ forms a basis of $B_{i-1}$,
\item
$\mathbf{h}_i:=\{h_{i,1},h_{i,2},\dots,h_{i,n_i}\}$ be a basis of $H_{i}$, and
\item
$\tilde{\mathbf{h}}_{i}:=\{\tilde{h}_{i,1},\tilde{h}_{i,2},\dots,\tilde{h}_{i,n_i}\}$ be a set of vectors such that $\tilde{h}_{i,k}$ is a lift of $h_{i,k}$ in $Z_{i}$.
\end{itemize}
Then it can be easily seen that $\partial_{i+1}(\mathbf{b}_{i+1})\cup\tilde{\mathbf{h}}_{i}\cup\mathbf{b}_{i}$ forms a basis of $C_{i}$.
\par
Denote by $[\mathbf{u}\mid\mathbf{v}]$ the determinant of the basis change matrix from $\mathbf{u}$ to $\mathbf{v}$ for two bases $\mathbf{u}$ and $\mathbf{v}$ of a vector space.
Put
\begin{equation}\label{eq:Reidemeister_def}
  \operatorname{Tor}(C_{\ast},\mathbf{c}_{\ast},\mathbf{h}_{\ast})
  :=
  \prod_{i}
  \left[
    \partial_{i+1}(\mathbf{b}_{i+1})\cup\tilde{\mathbf{h}}_{i}\cup\mathbf{b}_{i}\Bigm|\mathbf{c}_{i}
  \right]^{(-1)^{i+1}}
  \in\C^{\ast}.
\end{equation}
Note that this does not depend on the choices of $\mathbf{b}_{i}$ nor the choices of lifts of $\mathbf{h}_{i}$ (see for example \cite{Turaev:2001}).
\par
Now we consider the case where $M$ is a knot complement $S^3\setminus\Int{N(K)}$.
\par
Let $\langle x_1,x_2,\dots,x_n\mid r_1,r_2,\dots,r_{n-1}\rangle$ be a Wirtinger presentation of $\Pi:=\pi_1(M)$.
Since the Reidemeister torsion depends only on the simple homotopy type (see \cite[p.~10, Remarque (b)]{Porti:MAMCAU1997},\cite{Milnor:BULAM11966}) and the Whitehead group of a knot exterior is trivial \cite{Waldhausen:ANNMA21978}, we can calculate the torsion regarding $M$ as a CW-complex $\mathcal{K}$ with one $0$-cell $v$, $n$ $1$-cells $x_1,x_2,\dots,x_{n}$, and $(n-1)$ $2$-cells $f_1,f_2,\dots,f_{n-1}$.
Here the $2$-cell $f_j$ is attached to the bouquet $v\cup\left(\bigcup_{i=1}^{n}x_i\right)$ so that its boundary $\partial f_j$ presents the word $r_j$.
\par
Since our CW-complex is $2$-dimensional the Reidemeister torsion becomes
\begin{equation*}
  \frac{
  \left[
    \partial_{2}(\mathbf{b}_{2})\cup\tilde{\mathbf{h}}_{1}\cup\mathbf{b}_{1}\Bigm|\mathbf{c}_{1}
  \right]}
  {
  \left[
    \partial_{1}(\mathbf{b}_{1})\cup\tilde{\mathbf{h}}_{0}\Bigm|\mathbf{c}_{0}
  \right]
  \left[
    \tilde{\mathbf{h}}_{2}\cup\mathbf{b}_{2}\Bigm|\mathbf{c}_{2}
  \right]}.
\end{equation*}
See Figure~\ref{fig:chain_complex}.
\begin{figure}[h]
  \includegraphics[scale=0.3]{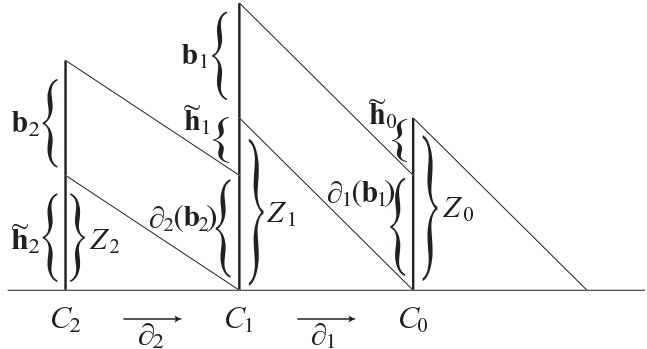}
  \caption{Chain complex and its basis}
  \label{fig:chain_complex}
\end{figure}
\par
Following \cite[D\'efinition~0.4]{Porti:MAMCAU1997}, we choose geometric bases for $C_{i}$ as follows.
Let $\tilde{v}$ be a lift of the $0$-cell $v$, $\tilde{x}_j$ a lift of $x_j$ ($j=1,2,\dots,n$), and $\tilde{f}_k$ a lift of $f_k$ ($k=1,2,\dots,n-1$).
If we put $E:=\begin{pmatrix}0&1\\0&0\end{pmatrix}$, $H:=\begin{pmatrix}1&0\\0&-1\end{pmatrix}$ and $F:=\begin{pmatrix}0&0\\1&0\end{pmatrix}$, then $\{E,H,F\}$ forms a basis of $\sl_2(\C)$.
Note that we can use any basis of $\sl_2(\C)$ to define geometric bases since the Euler characteristic of a knot complement is zero (see \cite[D\'efinition~0.5]{Porti:MAMCAU1997}).
We choose $\{\tilde{v}\otimes{E},\tilde{v}\otimes{H},\tilde{v}\otimes{F}\}$ as a geometric basis for $C_0$, $\{\tilde{x}_1\otimes{E},\tilde{x}_1\otimes{H},\tilde{x}_1\otimes{F},\tilde{x}_2\otimes{E},\tilde{x}_2\otimes{H},\tilde{x}_2\otimes{F},\dots,\tilde{x}_n\otimes{E},\tilde{x}_n\otimes{H},\tilde{x}_n\otimes{F}\}$ as a geometric basis for $C_1$, and $\{\tilde{f}_1\otimes{E},\tilde{f}_1\otimes{H},\tilde{f}_1\otimes{F},\tilde{f}_2\otimes{E},\tilde{f}_2\otimes{H},\tilde{f}_2\otimes{F},\dots,\tilde{f}_{n-1}\otimes{E},\tilde{f}_{n-1}\otimes{H},\tilde{f}_{n-1}\otimes{F}\}$ as a geometric basis for $C_2$.
\par
With respect to these bases the differentials $\partial_2$ and $\partial_1$ are given by the Fox free differential calculus (see \cite[Chapter~11]{Lickorish:1997} for example).
\par
Let $G_n$ be the free group generated by $\{x_1,x_2,\dots,x_n\}$ and $\Phi\colon G_n\to\Pi$ the quotient map sending $x_i\in G_n$ to $x_i\in\Pi$.
Let $\Ad{\rho}(x_i)$ be the $3\times3$ matrix
\begin{equation*}
  \begin{pmatrix}
    \Ad{\rho(\Phi(x_i))}(E)_{1}&\Ad{\rho(\Phi(x_i))}(H)_{1}&\Ad{\rho(\Phi(x_i))}(F)_{1} \\
    \Ad{\rho(\Phi(x_i))}(E)_{2}&\Ad{\rho(\phi(x_i))}(H)_{2}&\Ad{\rho(\Phi(x_i))}(F)_{2} \\
    \Ad{\rho(\Phi(x_i))}(E)_{3}&\Ad{\rho(\Phi(x_i))}(H)_{3}&\Ad{\rho(\Phi(x_i))}(F)_{3}
  \end{pmatrix},
\end{equation*}
where 
\begin{align*}
  \Ad{\rho(\Phi(x_i))}(E)
  &=
  \Ad{\rho(\Phi(x_i))}(E)_{1}E+\Ad{\rho(\Phi(x_i))}(E)_{2}H+\Ad{\rho(\Phi(x_i))}(E)_{3}F,
  \\
  \Ad{\rho(\Phi(x_i))}(H)
  &=
  \Ad{\rho(\Phi(x_i))}(H)_{1}E+\Ad{\rho(\Phi(x_i))}(H)_{2}H+\Ad{\rho(\Phi(x_i))}(H)_{3}F,
  \\
  \Ad{\rho(\Phi(x_i))}(F)
  &=
  \Ad{\rho(\Phi(x_i))}(F)_{1}E+\Ad{\rho(\Phi(x_i))}(F)_{2}H+\Ad{\rho(\Phi(x_i))}(F)_{3}F.
\end{align*}
\par
Denoting by $\frac{\partial}{\partial\,x_j}\colon \Z[G_n]\to\Z[G_n]$ the Fox derivative \cite{Fox:free_differential_calculus_I}, the differential $\partial_2$ is given by the following $3(n-1)\times3n$ matrix:
\begin{equation}\label{eq:differential_2}
  \partial_2
  =
  \begin{pmatrix}
    \Ad{\rho}\left(\frac{\partial\,r_1}{\partial\,x_1}\right)&\cdots
    &\Ad{\rho}\left(\frac{\partial\,r_{n-1}}{\partial\,x_1}\right)
    \\
    \vdots&\ddots&\vdots
    \\
    \Ad{\rho}\left(\frac{\partial\,r_1}{\partial\,x_n}\right)&\cdots
    &\Ad{\rho}\left(\frac{\partial\,r_{n-1}}{\partial\,x_n}\right)
  \end{pmatrix}.
\end{equation}
Here $\Phi$ is regarded as a homomorphism $\Z[G_n]\to\Z[\Pi]$.
The differential $\partial_1$ is given by the following $3\times3n$ matrix:
\begin{equation}\label{eq:differential_1}
  \partial_1
  =
  \begin{pmatrix}
    \Ad{\rho}\left(x_1-1\right)&\cdots&\Ad{\rho}\left(x_n-1\right)
  \end{pmatrix}.
\end{equation}

\section{Fundamental group}\label{sec:pi1}
Let $T(2,2a+1)^{(2,2b+1)}$ be the $(2,2b+1)$-cable of the torus knot of type $(2,2a+1)$ (Figure~\ref{fig:iterated_torus_knot}).
Denote by $E:=S^3\setminus\Int{N\left(T(2,2a+1)^{(2,2b+1)}\right)}$ the complement of the interior of the regular neighborhood of $T(2,2a+1)^{(2,2b+1)}$, where $N$ means the regular neighborhood in the three-sphere $S^3$ and $\Int$ means the interior.
Then $E$ can be decomposed into $\pattern$ and $C$, where $\pattern:=S^3\setminus\Int{N(T(2,2a+1))}$ and $C$ is the knot in a solid torus depicted in Figure~\ref{fig:twist_2b+1}.
Put $S:=\pattern\cap C=\partial{\pattern}=\partial{C}$.
In other words $T(2,2a+1)^{(2,2b+1)}$ is the $(2,2b+1)$-cable of $T(2,2a+1)$, or the satellite knot of $T(2,2a+1)$ with companion the torus knot of type $(2,2b+1)$ on $S$.
\par
We calculate the fundamental group of $E$.
For more details see \cite[\S1.3]{Murakami:ACTMV2014}.
\par
Let $x$ and $y$ be the generators of the fundamental group $\pi_1(C)$ of $C$ as indicated in Figure~\ref{fig:torus_knot_pi1}, where the basepoint is on $S=\partial{C}$.
\begin{figure}[h!]
  \includegraphics[scale=0.3]{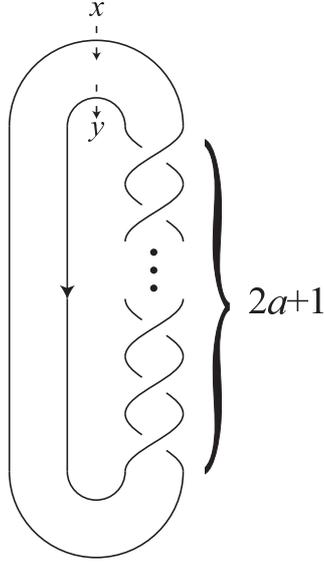}
  \caption{Generators of $\pi_1(C)$}
  \label{fig:torus_knot_pi1}
\end{figure}
Then we can easily see that $\pi_1(C)$ is presented as follows:
\begin{equation}\label{eq:pi1_C}
  \pi_1(C)
  =
  \langle
    x,y\mid
    (xy)^ax=y(xy)^a
  \rangle.
\end{equation}
Let $p$, $q$, and $r$ be the generators of $\pi_1(\pattern)$ as indicated in Figure~\ref{fig:pattern_pi1} with the basepoint on $S=\partial{\pattern}$ as before.
\begin{figure}[h!]
  \includegraphics[scale=0.25]{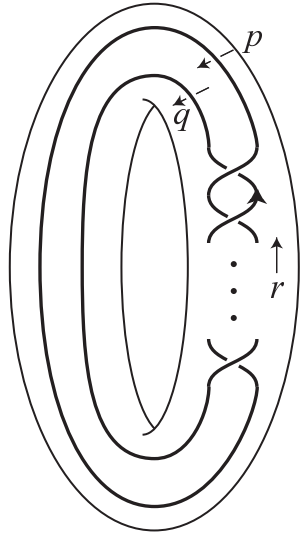}
  \caption{Generators of $\pi_1(\pattern)$}
  \label{fig:pattern_pi1}
\end{figure}
Then we have
\begin{equation*}
  \pi_1(\pattern)
  =
  \langle
    p,q,r\mid
    (pq)r=r(pq), r(pq)^bp=q(pq)^br
  \rangle.
\end{equation*}
If we give left $b$ full-twists to the solid torus, we have Figure~\ref{fig:pattern_twisted}.
\begin{figure}[h!]
  \includegraphics[scale=0.25]{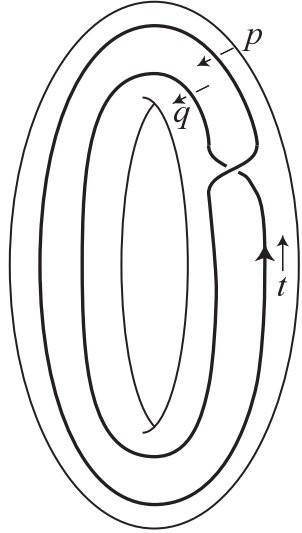}
  \caption{Twist the solid torus in Figure~\ref{fig:pattern_pi1} $b$ times counterclockwise}
  \label{fig:pattern_twisted}
\end{figure}
Then the element $t$ indicated in Figure~\ref{fig:pattern_twisted} equals $r(pq)^{b}$.
From this picture it is not hard to see $q=tpt^{-1}$.
Therefore we can reduce the presentation of $\pi_1(\pattern)$ as follows:
\begin{equation}\label{eq:pi1_pattern}
\begin{split}
  &\pi_1(\pattern)
  \\
  =&
  \left\langle
    p,q,r,t\mid
    (pq)r=r(pq), r(pq)^bp=q(pq)^br,r=t(pq)^{-b},q=tpt^{-1}
  \right\rangle
  \\
  =&
  \left\langle
    p,t\Bigm|
    ptpt^{-1}t\left(ptpt^{-1}\right)^{-b}=t\left(ptpt^{-1}\right)^{-b}ptpt^{-1},
  \right.
  \\
  &\qquad\quad
  \left.
    tp=tpt^{-1}(ptpt^{-1})^bt\left(ptpt^{-1}\right)^{-b}
  \right\rangle
  \\
  =&
  \left\langle
    p,t\mid ptpt=tptp
  \right\rangle.
\end{split}
\end{equation}
Note that the last presentation is the Wirtinger presentation of the torus link of type $(2,4)$ as depicted in Figure~\ref{fig:2_4}.
In fact, it can be easily seen that the complement of the knot in the solid torus of Figure~\ref{fig:pattern_twisted} is homeomorphic to the complement of the $(2,4)$ torus link.
\begin{figure}[h!]
  \includegraphics[scale=0.3]{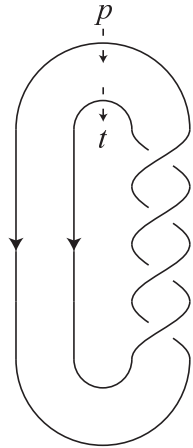}
  \caption{$(2,4)$ torus link}
  \label{fig:2_4}
\end{figure}
\par
Let $i^{C}\colon S\hookrightarrow C$, $i^{\pattern}\colon S\hookrightarrow\pattern$, $j^{C}\colon C\hookrightarrow E$, and $j^{\pattern}\colon D\hookrightarrow E$ be the inclusion maps as in \eqref{eq:inclusions}.
\begin{equation}\label{eq:inclusions}
\setlength\unitlength{3mm}
\raisebox{-14mm}{
\begin{picture}(20,10)(0,0)
  \put(0,5){\makebox(0,0){$S$}}
  \put(10,10){\makebox(0,0){$C$}}
  \put(10,0){\makebox(0,0){$\pattern$}}
  \put(20,5){\makebox(0,0){$E$}}
  \put(5,9){\makebox(0,0){$i^{C}$}}
  \put(5,1){\makebox(0,0){$i^{\pattern}$}}
  \put(15,9){\makebox(0,0){$j^{C}$}}
  \put(15,1){\makebox(0,0){$j^{\pattern}$}}
  \put(1,6){\vector(2,1){8}}
  \put(1,4){\vector(2,-1){8}}
  \put(11,10){\vector(2,-1){8}}
  \put(11,0){\vector(2,1){8}}
\end{picture}
}
\end{equation}
If we denote by $\lambda_{C}$ the preferred longitude of $T(2,2a+1)$, then its homotopy class is given by
\begin{equation*}
  \lambda_{C}
  =
  y(xy)^{2a}x^{-4a-1},
\end{equation*}
which is identified with $r=t\left(ptpt^{-1}\right)^{-b}\in\pi_1(\pattern)$.
The homotopy class of the meridian $\mu_{C}$ is given by $x$.
\par
So from \eqref{eq:pi1_C}, \eqref{eq:pi1_pattern} and van Kampen's theorem we have
\begin{equation*}
\begin{split}
  &\pi_1(E)
  \\
  =&
  \pi_1(C)\underset{\pi_1(S)}{\Asterisk}\pi_1(\pattern)
  \\
  =&
  \left\langle
    x,y,p,t\Bigm|
    (xy)^ax=y(xy)^a,ptpt=tptp,y(xy)^{2a}x^{-4a-1}=t\left(ptpt^{-1}\right)^{-b},
  \right.
  \\
  &\qquad\qquad\quad
  \left.\vphantom{\bigm|}
  x=ptpt^{-1}
  \right\rangle.
\end{split}
\end{equation*}
Note that we do not need the second relation because from the third and the fourth relations we have $t=y(xy)^{2a}x^{-4a-1}x^{b}=y(xy)^{2a}x^{-4a+b-1}$ and $ptp=xt=(xy)^{2a+1}x^{-4a+b-1}$, and so we have
\begin{equation*}
\begin{split}
  &t\cdot ptp\cdot t^{-1}\cdot(ptp)^{-1}
  \\
  =&
  y(xy)^{2a}x^{-4a+b-1}\cdot(xy)^{2a+1}x^{-4a+b-1}\cdot x^{4a-b+1}(xy)^{-2a}y^{-1}\cdot x^{4a-b+1}(xy)^{-2a-1}
  \\
  =&
  y(xy)^{2a}x(xy)^{-2a-1}
  \\
  &\text{(from the first relation)}
  \\
  =&
  (xy)^{a}x(xy)^{a}(xy)^{-a}y(xy)^{-a-1}
  \\
  =&
  1.
\end{split}
\end{equation*}
Therefore we have the following presentation with deficiency one.
\begin{equation*}
  \pi_1(E)
  =
  \langle
    x,y,p,t\mid
    (xy)^ax=y(xy)^a,y(xy)^{2a}x^{-4a-1}=t\left(ptpt^{-1}\right)^{-b},x=ptpt^{-1}
  \rangle.
\end{equation*}
\par
We can also calculate the homotopy class of the longitude $\lambda$ of $T(2,2a+1)^{(2,2b+1)}$ as follows (see \cite{Murakami:ACTMV2014}).
\begin{equation}\label{eq:lambda}
  \lambda
  =
  r(pq)^bpq^{-b}r(pq)^bp^{-3b-1}
  =
  \lambda_{C}x^bpq^{-b}\lambda_{C}x^bp^{-3b-1}.
\end{equation}
in $\pi_1(E)$.
The homotopy class of the meridian of $T(2,2a+1)^{(2,2b+1)}$ is $p$.
\section{Representation}\label{sec:rep}
In this section we study representations of $\pi_1(E)$ into $\SL(2;C)$.
For a representation $\rho\colon\pi_1(E)\to\SL(2;C)$, we denote its restriction to $\pi_1(C)$ ($\pi_1(\pattern)$, respectively) by $\rho_{C}$ ($\rho_{\pattern}$, respectively).
In the following, see \cite[\S1.3]{Murakami:ACTMV2014} again for details.
\par
We consider the following four cases:
(1) the case where both $\Im\rho_{C}$ and $\Im\rho_{\pattern}$ are Abelian,
(2) the case where $\Im\rho_{C}$ is Abelian and $\Im\rho_{\pattern}$ is non-Abelian,
(3) the case where $\Im\rho_{C}$ is non-Abelian and $\Im\rho_{\pattern}$ is Abelian, and
(4) the case where both $\Im\rho_{C}$ and $\Im\rho_{\pattern}$ are non-Abelian.
\subsection{Both $\Im\rho_{C}$ and $\Im\rho_{\pattern}$ are Abelian}\label{subsec:rep_AA}
For a complex number $z$, we put
\begin{align*}
  \rho(p)
  &=
  \rho(q)
  :=
  \begin{pmatrix}
    z&0\\
    0&z^{-1}
  \end{pmatrix},
  \\
  \rho(x)
  &=
  \rho(y)
  :=
  \rho(p)\rho(q)
  =
  \begin{pmatrix}
    z^2&0 \\
    0  &z^{-2}
  \end{pmatrix}.
\end{align*}
Since the longitude $\lambda_{C}$ is null-homologous, we have $\rho(r)=\rho(\lambda_{C})=I_2$, where $I_2$ is the two by two identity matrix.
Therefore
\begin{equation*}
  \rho(t)
  =
  \rho\left(r(pq)^b\right)
  =
  \begin{pmatrix}
    z^{2b}&0 \\
    0     &z^{-2b}
  \end{pmatrix}.
\end{equation*}
\par
Note that the longitude $\lambda$ of $T(2,2a+1)^{(2,2b+1)}$ is sent to $I_2$.
\par
We denote by $\rho^{\rm{AA}}_{\xi}$ be the representation $\rho$ with $z=\exp(\xi/2)$.
\subsection{$\Im\rho_{C}$ is Abelian and $\Im\rho_{\pattern}$ is non-Abelian}\label{subsec:rep_AN}
For a complex number $z$, if we define
\begin{align*}
  \rho(p)
  &:=
  \begin{pmatrix}
    z&1\\
    0&z^{-1}
  \end{pmatrix},
  \\
  \rho(q)
  &:=
  \begin{pmatrix}
    z                                &0 \\
    \omega_2+\omega_2^{-1}-z^2-z^{-2}&z^{-1}
  \end{pmatrix},
  \\
  \rho(x)
  &=
  \rho(y)
  :=
  \rho(p)\rho(q)
  =
  \theta_1^{-1}
  \begin{pmatrix}
    \omega_2&z^{-1} \\
    0       &\omega_2^{-1}
  \end{pmatrix}
  \theta_1
\end{align*}
with $\omega_2^{2b+1}=-1$ ($\omega_2\ne-1$) and
\begin{equation*}
  \theta_1
  :=
  \begin{pmatrix}
    1                    &0 \\
    \omega_2^{-1}z-z^{-1}&1
  \end{pmatrix},
\end{equation*}
then $\rho$ is well defined.
Since the longitude $\lambda_{C}$ is null-homologous, we have $\rho(r)=\rho(\lambda_{C})=I_2$, where $I_2$ is the two by two identity matrix.
Therefore
\begin{equation*}
  \rho(t)
  =
  \rho\left(r(pq)^b\right)
  =
  \theta_1^{-1}
  \begin{pmatrix}
    \omega_2^{b}&\frac{\omega_2^{b}-\omega_2^{-b}}{\omega_2-\omega_2^{-1}}z^{-1}\\
    0           &\omega_2^{-b}
  \end{pmatrix}
  \theta_1.
\end{equation*}
\par
The longitude $\lambda$ of $T(2,2a+1)^{(2,2b+1)}$ is sent to
\begin{equation*}
  \rho(\lambda)
  =
  \begin{pmatrix}
    -z^{-(2b+1)}&\frac{z^{2b+1}-z^{-(2b+1)}}{z-z^{-1}} \\
    0           &-z^{2b+1}
  \end{pmatrix}.
\end{equation*}
\par
We denote by $\rho^{\rm{AN}}_{\xi,j}$ be the representation $\rho$ with $z=\exp(\xi/2)$ and $\omega_2=\exp\left(\frac{(2j+1)\pi\sqrt{-1}}{2b+1}\right)$ ($j=0,1,\dots,b-1$).
\subsection{$\Im\rho_{C}$ is non-Abelian and $\Im\rho_{\pattern}$ is Abelian}\label{subsec:rep_NA}
If we define
\begin{align*}
  \rho(p)
  &=
  \rho(q)
  :=
  \begin{pmatrix}
    z&\frac{1}{z+z^{-1}} \\
    0&z^{-1}
  \end{pmatrix},
  \\
  \rho(x)
  &:=
  \begin{pmatrix}
    z^2&1 \\
    0&z^{-2}
  \end{pmatrix},
  \\
  \rho(y)
  &:=
  \begin{pmatrix}
    z^2                              &0 \\
    \omega_1+\omega_1^{-1}-z^4-z^{-4}&z^{-2}
  \end{pmatrix}
\end{align*}
with $\omega_1^{2a+1}=-1$ ($\omega_1\ne-1$), then $\rho$ is well-defined.
Note that the longitude $\lambda_{C}$ of $C$ is presented as
\begin{equation*}
  \rho(\lambda_{C})
  =
  \rho(y(xy)^{2a}x^{-4a-1})
  =
  \begin{pmatrix}
    -z^{-8a-4}&\frac{z^{8a+4}-z^{-8a-4}}{z^2-z^{-2}} \\
    0         &-z^{8a+4}
  \end{pmatrix}
  =
  -\rho(p)^{-8a-4},
\end{equation*}
which equals $\rho(r)$.
We also have
\begin{equation*}
  \rho(t)
  =
  \rho(r(pq)^b)
  =
  \begin{pmatrix}
    -z^{-8a+2b-4}&\frac{z^{8a-2b+4}-z^{-8a+2b-4})}{z^2-z^{-2}} \\
    0            &-z^{8a-2b+4}
  \end{pmatrix}
  =
  -\rho(p)^{-8a+2b-4}.
\end{equation*}
Note that if we put
\begin{equation*}
  \theta_2
  :=
  \begin{pmatrix}
    1&\frac{z^2}{z^4-1} \\
    0&1
  \end{pmatrix},
\end{equation*}
we have
\begin{align*}
  \rho(p)
  &=
  \theta_2^{-1}
  \begin{pmatrix}z&0\\0&z^{-1}\end{pmatrix}
  \theta_2,
  \\
  \rho(\lambda_{C})
  &=
  \theta_2^{-1}
  \begin{pmatrix}-z^{-8a-4}&0\\0&-z^{8a+4}\end{pmatrix}
  \theta_2,
  \\
  \rho(t)
  &=
  \theta_2^{-1}
  \begin{pmatrix}-z^{-8a+2a-4}&0\\0&-z^{8a-2b+4}\end{pmatrix}
  \theta_2.
\end{align*}
\par
The longitude $\lambda$ of $T(2,2a+1)^{(2,2b+1)}$ is sent to
\begin{equation*}
  \rho(\lambda)
  =
  \begin{pmatrix}
    z^{-8(2a+1)+(2b+1)}&\frac{z^{-8(2a+1)+(2b+1)}-z^{8(2a+1)-(2b+1)}}{z^2-z^{-2}} \\
    0                  &z^{8(2a+1)-(2b+1)}
  \end{pmatrix}
  =
  \rho(p)^{-8(2a+1)+(2b+1)}.
\end{equation*}
\par
We denote by $\rho^{\rm{NA}}_{\xi,k}$ be the representation $\rho$ with $z=\exp(\xi/2)$ and $\omega_1:=\exp\left(\frac{(2k+1)\pi\sqrt{-1}}{2a+1}\right)$ ($k=0,1,\dots,a-1$).
\subsection{Both $\Im\rho_{C}$ and $\Im\rho_{\pattern}$ are non-Abelian}\label{subsec:rep_NN}
If we define
\begin{align*}
  \rho(p)
  &:=
  \begin{pmatrix}
    z&1\\
    0&z^{-1}
  \end{pmatrix},
  \\
  \rho(q)
  &:=
  \begin{pmatrix}
    z                                &0 \\
    \omega_3+\omega_3^{-1}-z^2-z^{-2}&z^{-1}
  \end{pmatrix},
  \\
  \rho(x)
  &:=
  \rho(pq)
  =
  \theta_3^{-1}
  \begin{pmatrix}
    \omega_3&1 \\
    0       &\omega_3^{-1}
  \end{pmatrix}
  \theta_3
  =
  \tilde{\theta}_1^{-1}
  \begin{pmatrix}
    \omega_3&z^{-1} \\
    0       &\omega_3^{-1}
  \end{pmatrix}
  \tilde{\theta}_1,
  \\
  \rho(y)
  &:=
  \theta_3^{-1}
  \begin{pmatrix}
    \omega_3&0 \\
    \omega_1+\omega_1^{-1}-\omega_3^2-\omega_3^{-2}&\omega_3^{-1}
  \end{pmatrix}
  \theta_3
\end{align*}
with $\omega_1^{2a+1}=-1$, $\omega_3^{2b+1-4(2a+1)}=-1$ ($\omega_3\ne-1$),
\begin{equation*}
  \tilde{\theta}_1
  :=
  \begin{pmatrix}
    1                    &0 \\
    \omega_3^{-1}z-z^{-1}&1
  \end{pmatrix},
\end{equation*}
and
\begin{equation*}
  \theta_3
  :=
  \begin{pmatrix}
    z^{1/2}&0 \\
    0      &z^{-1/2}
  \end{pmatrix}
  \tilde{\theta}_1
  =
  \begin{pmatrix}
    z^{1/2}                      &0 \\
    z^{1/2}\omega_3^{-1}-z^{-3/2}&z^{-1/2}
  \end{pmatrix}.
\end{equation*}
Note that $\tilde{\theta}_1$ is obtained from $\theta_1$ by replacing $\omega_2$ with $\omega_3$.
The images of $\lambda_{C}$  and $t$ are given as follows:
\begin{align*}
  \rho(\lambda_{C})
  &=
  \theta_3^{-1}
  \begin{pmatrix}
    -\omega_3^{-4a-2}&\dfrac{\omega_3^{4a+2}-\omega_3^{-4a-2}}{\omega_3-\omega_3^{-1}} \\[5mm]
    0                &-\omega_3^{4a+2}
  \end{pmatrix}
  \theta_3
  \\
  &=
  \tilde{\theta}_1^{-1}
  \begin{pmatrix}
    -\omega_3^{-4a-2}&\frac{\omega_3^{4a+2}-\omega_3^{-4a-2}}{\omega_3-\omega_3^{-1}}z^{-1} \\[5mm]
    0                &-\omega_3^{4a+2}
  \end{pmatrix}
  \tilde{\theta}_1
  \\
  \rho(t)
  &=
  \theta_3^{-1}
  \begin{pmatrix}
    -\omega_3^{-4a+b-2}&\dfrac{\omega_3^{4a-b+2}-\omega_3^{-4a+b-2}}{\omega_3-\omega_3^{-1}} \\[5mm]
    0                  &-\omega_3^{4a-b+2}
  \end{pmatrix}
  \theta_3
  \\
  &=
  \tilde{\theta}_1^{-1}
  \begin{pmatrix}
    \omega_3^{4a-b+1}&\dfrac{\omega_3^{4a-b+1}-\omega_3^{-4a+b-1}}{\omega_3-\omega_3^{-1}}z^{-1} \\[5mm]
    0                &\omega_3^{-4a+b-1}
  \end{pmatrix}
  \tilde{\theta}_1.
\end{align*}
\par
The image of the longitude $\lambda$ is given by
\begin{equation*}
  \begin{pmatrix}
    -z^{-4b-2}&\frac{z^{4b+2}-z^{4b-2}}{z-z^{-1}} \\
    0         &-z^{4b+2}
  \end{pmatrix}.
\end{equation*}
\par
We denote by $\rho^{\rm{NN}}_{\xi,l,m}$ be the representation $\rho$ with $z=\exp(\xi/2)$, $\omega_1:=\exp\left(\frac{(2m+1)\pi\sqrt{-1}}{2a+1}\right)$ ($m=0,1,\dots,a-1$) and $\omega_3:=\exp\left(\frac{(2l+1)\pi\sqrt{-1}}{2b+1-4(2a+1)}\right)$ ($l=0,1,\dots,b-4a-3$).
\section{Reidemeister torsion - Abelian case}
In this section we calculate the Reidemeister torsion of $E:=S^3\setminus\Int{N(T(2,2a+1)^{(2,2b+1)})}$ twisted by the adjoint action of an Abelian representation $\rho$.
We prove a more general formula of the Reidemeister torsion of an Abelian representation.
\par
Let $K$ be a knot in $S^3$ and put $E(K):=S^3\setminus\Int{N(K)}$.
Denote by $\pi_1(E(K))$ the fundamental group of $E(K)$ with basepoint in $\partial{E(K)}$.
We may assume that $\pi_1(E(K))$ has a Wirtinger presentation $\langle x_1,x_2,\dots,x_n\mid r_1,r_2,\dots,r_{n-1}\rangle$ so that if we Abelianize it then $x_i$ is sent to the generator $t\in\Z$ for $i=1,2,\dots,n$ (see for example \cite{Lickorish:1997}).
Here we regard $\Z$ as a multiplicative group $\langle t\mid-\rangle$.
\par
Let $\rho$ be the Abelian representation of $\pi_1(E(K))$ into $\SL(2;\C)$ sending every $x_i$ to $\begin{pmatrix}z&0\\0&z^{-1}\end{pmatrix}$ with $z\ne\pm1$.
Then we have the following proposition, which should be well known to experts.
Compare it with \cite[Theorem~4]{Milnor:ANNMA21962}.
\begin{prop}[Folklore]\label{prop:Alexander}
The Reidemester torsion $\Tor(E(K);\rho)$ of $E(K)$ twisted by the adjoint action of the representation $\rho$  associated with the meridian is given by
\begin{equation*}
  \Tor(E(K);\rho)
  =
  \pm
  \left(\frac{\Delta(K;z^2)}{z-z^{-1}}\right)^2,
\end{equation*}
where $\Delta(K;t)$ is the Alexander polynomial normalized so that $\Delta(K;1)=1$ and $\Delta(K;t^{-1})=\Delta(K;t)$.
\end{prop}
\begin{cor}
Let $\rho^{\rm{AA}}_{\xi}$ be the representation defined in Subsection~\ref{subsec:rep_AA} with $z=\exp(\xi/2)$.
Then we have
\begin{equation*}
  \Tor(E;\rho^{\rm{AA}}_{\xi})
  =
  \pm
  \left(\frac{\Delta(K;\exp\xi)}{2\sinh(\xi/2)}\right)^2
  =
  \pm\tau_0(\xi)^{-2}.
\end{equation*}
\end{cor}
\begin{proof}[Proof of Proposition~\ref{prop:Alexander}]
First note that $E(K)$ is simple homotopy equivalent to a cell complex with one $0$-cell $v$ (the basepoint), $n$ $1$-cells $x_1,x_2,\dots,x_n$, and $(n-1)$ $2$-cells $f_1,f_2,\dots,f_{n-1}$.
Here the $2$-cell $f_j$ is attached to the bouquet $v\cup\left(\bigcup_{i=1}^{n}x_i\right)$ so that its boundary $\partial{f_j}$ presents the word $r_i$.
\par
Then as described in Section~\ref{sec:definition}, the differentials of the chain complex $\{0\}\rightarrow C_2(E(K);\rho)\xrightarrow{\partial_2}C_1(E(K);\rho)\xrightarrow{\partial_1}C_0(E(K);\rho)\rightarrow\{0\}$ is given as follows:
\begin{align*}
  \partial_2
  &=
  \begin{pmatrix}
    \Ad{\rho}\left(\frac{\partial r_1}{\partial x_1}\right)&\cdots&
      \Ad{\rho}\left(\frac{\partial r_{n-1}}{\partial x_1}\right)
    \\
    \vdots&\ddots&\vdots
    \\
    \Ad{\rho}\left(\frac{\partial r_1}{\partial x_n}\right)&\cdots&
      \Ad{\rho}\left(\frac{\partial r_{n-1}}{\partial x_n}\right)
  \end{pmatrix},
  \\
  \partial_1
  &=
  \begin{pmatrix}
    \Ad{\rho}(x_1-1)&\Ad{\rho}(x_2-1)&\cdots&\Ad{\rho}(x_n-1)
  \end{pmatrix}
  \\
  &=
  \begin{pmatrix}X-I_3&X-I_3&\cdots&X-I_3\end{pmatrix}
  \\
  &=
  \left(
    \begin{array}{ccc:ccc:c:ccc}
      z^{-2}-1&0&0      &z^{-2}-1&0&0      &      &z^{-2}-1&0&0      \\
      0       &0&0      &0       &0&0      &\cdots&0       &0&0      \\
      0       &0&z^{2}-1&0       &0&z^{2}-1&      &0       &0&z^{2}-1
    \end{array}
  \right).
\end{align*}
Here $X$ is the following three by three matrix:
\begin{equation*}
\begin{split}
  X
  &:=
  \Ad{\rho}(x_i)
  \\
  &=
  \begin{pmatrix}
    \Ad{\rho(\Phi(x_i))}(E)_1&\Ad{\rho(\Phi(x_i))}(H)_1&\Ad{\rho(\Phi(x_i))}(F)_1 \\
    \Ad{\rho(\Phi(x_i))}(E)_2&\Ad{\rho(\Phi(x_i))}(H)_1&\Ad{\rho(\Phi(x_i))}(F)_2 \\
    \Ad{\rho(\Phi(x_i))}(E)_3&\Ad{\rho(\Phi(x_i))}(H)_1&\Ad{\rho(\Phi(x_i))}(F)_3 \\
  \end{pmatrix}
  \\
  &=
  \begin{pmatrix}
    z^{-2}&0&0 \\
    0     &1&0 \\
    0     &0&z^{2}
  \end{pmatrix}
\end{split}
\end{equation*}
since $\rho(x_j)=\begin{pmatrix}z&0\\0&z^{-1}\end{pmatrix}$ for any $j$ and
\begin{align*}
  \rho(x_i)^{-1}E\rho(x_i)
  &=
  \begin{pmatrix}
    0&z^{-2}\\
    0&0
  \end{pmatrix},
  \\
  \rho(x_i)^{-1}H\rho(x_i)
  &=
  \begin{pmatrix}
    1&0 \\
    0&-1
  \end{pmatrix},
  \\
  \rho(x_i)^{-1}F\rho(x_i)
  &=
  \begin{pmatrix}
    0    &0 \\
    z^{2}&0
  \end{pmatrix}.
\end{align*}
\par
Note that if we let $A(t):=\alpha\left(\dfrac{\partial r_j}{\partial x_i}\right)_{\substack{i=1,2,\dots,n\\j=1,2,\dots,n-1\hfill}}$ be the Alexander matrix of $K$ with $\alpha\colon\Z[\pi_1(E(K))]\to\Z[t,t^{-1}]$ the homomorphism induced by the Abelianization, then $\partial_2$ is given by $A(X)$, where we replace $t$ in $A(t)$ with the $3\times3$ matrix $X$ (and $1$ with $I_3$).
Note also that if we denote by $\tilde{A}(t)$ the $(n-1)\times(n-1)$ matrix obtained from $A(t)$ by removing the first row, then $\det\tilde{A}(t)$ equals the (unnormalized) Alexander polynomial $\tilde{\Delta}(K;t)$ of $K$.
Since $X$ is diagonal with diagonal entries $z^{-2}$, $1$, and $z^2$, we have
\begin{equation*}
\begin{split}
  \det\tilde{A}(X)
  &=
  \det\tilde{A}(z^{-2})\det\tilde{A}(1)\det\tilde{A}(z^2)
  \\
  &=
  \tilde{\Delta}(K;z^{-2})\tilde{\Delta}(K;1)\tilde{\Delta}(K;z^2)
  \\
  &=
  \pm\Delta(K;z^2)^2.
\end{split}
\end{equation*}
\par
It can be easily seen that $\Ker\partial_1=\C^{3n-2}$, generated by $\{\vec{e}_2,\vec{e}_{5},\dots,\vec{e}_{3n-1},\vec{e}_1-\vec{e}_4,\vec{e}_4-\vec{e}_7,\dots,\vec{e}_{3n-5}-\vec{e}_{3n-2},\vec{e}_3-\vec{e}_6,\vec{e}_6-\vec{e}_9,\dots,\vec{e}_{3n-3}-\vec{e}_{3n}\}$, where $\vec{e}_k\in C_1(E(K);\rho)$ is the $3n\times1$ matrix with $1$ at $(k,1)$-entry and $0$ otherwise, and that $\Im\partial_1=\C^2$, generated by $\left\{\begin{pmatrix}1\\0\\0\end{pmatrix},\begin{pmatrix}0\\0\\1\end{pmatrix}\right\}$.
Since the determinant of the $3(n-1)\times3(n-1)$ matrix $\tilde{A}(X)$ is not zero, the rank of $A(X)$ is $3n-3$, $\partial_2$ is injective, and so $\Im\partial_2=\C^{3n-3}$.
\par
Therefore we have
\begin{alignat*}{2}
  H_2(C;\rho)
  &=
  \{0\},
  &&
  \\
  H_1(C;\rho)
  &=
  \C
  &\quad&\text{(generated by $[\tilde{x}_1]\otimes\begin{pmatrix}0\\1\\0\end{pmatrix}$)},
  \\
  H_0(C;\rho)
  &=
  \C
  &\quad&\text{(generated by $[\tilde{v}]\otimes\begin{pmatrix}0\\1\\0\end{pmatrix}$)}.
\end{alignat*}
\par
Now we calculate the Reidemeister torsion of $E(K)$.
From the calculation above if we put
\begin{align*}
  \mathbf{b}_2
  &:=
  \left\langle
    \underbrace{
    \begin{pmatrix}1\\0\\\vdots\\0\end{pmatrix},
    \begin{pmatrix}0\\1\\\vdots\\0\end{pmatrix},
    \cdots,
    \begin{pmatrix}0\\0\\\vdots\\1\end{pmatrix}}_{3n-3}
  \right\rangle,
  \\
  \tilde{\mathbf{h}}_1
  &:=
  \left\langle\vec{e}_2\right\rangle,
  \\
  \mathbf{b}_1
  &:=
  \left\langle\vec{e}_1,\vec{e}_3\right\rangle,
  \\
  \tilde{\mathbf{h}}_0
  &:=
  \left\langle
    \begin{pmatrix}0\\1\\0\end{pmatrix}
  \right\rangle,
\end{align*}
we see that $\mathbf{b}_2$ forms a basis of $C_2(E(K);\rho)$, that $\mathbf{b}_1\cup\tilde{\mathbf{h}}_1\cup\partial_2(\mathbf{b}_2)$ forms a basis of $C_1(E(K);\rho)$, and that $\tilde{\mathbf{h}}_0\cup\partial_1(\mathbf{b}_1)$ forms a basis of $C_0(E(K);\rho)$.
Therefore the Reidemeister torsion of $E(K)$ associated with these bases is
\begin{equation*}
\begin{split}
  \Tor(E(K);\rho)
  &=\pm
  \frac{
  \left|
    \begin{array}{ccc:c}
      1&0&0& \\
      0&0&1& \\
      0&1&0& \\
      0&0&0& \\
      0&0&0& \\
      0&0&0& \\
      &\raisebox{3mm}{\vdots}&&\raisebox{10mm}{\mbox{\smash{\huge{$A(X)$}}}}\\
      0&0&0& \\
      0&0&0& \\
      0&0&0&
    \end{array}
  \right|
  }
  {
  \det(I_{3n-3})
  \begin{vmatrix}
    0&z^{-2}-1&0      \\
    1&0       &0      \\
    0&0       &z^{2}-1
  \end{vmatrix}}
  \\
  &=
  \pm
  \left(\frac{\Delta(K;z^2)}{(z-z^{-1})}\right)^2
\end{split}
\end{equation*}
as required.
\par
Note that we use $[\tilde{\mu}_{K}\otimes H]$ as the generator of $H_1(C;\rho)$ and $[\tilde{v}\otimes H]$ as the generator of $H_0(C;\rho)$.
Here $H$ can be characterized as the vector that is invariant under the adjoint action of $\rho(\mu_{K})$.
\end{proof}

\section{Reidemeister torsion - Non-Abelian case}
This section is devoted to the calculation of the Reidemeister torsions twisted by the adjoint actions of non-Abelian representations.
To do this we will use the following formula \cite[Proposition~0.11]{Porti:MAMCAU1997}.
\begin{equation}\label{eq:MV_torsion}
  \Tor(E;\rho)
  =
  \frac{\Tor(C;\rho)\Tor(\pattern;\rho)}{\Tor(S;\rho)\Tor(\mathcal{H}_{\ast})},
\end{equation}
where $\mathcal{H}_{\ast}$ is the following Mayer--Vietoris exact sequence associated with the decomposition of $E=C\cup_{S}D$:
\begin{equation}\label{eq:MV_exact}
\begin{CD}
  @>\delta_3>>H_2(S;\rho)
  @>\varphi_2>>H_2(C;\rho)\oplus H_2(\pattern;\rho)
  @>\psi_2>>H_2(E;\rho)
  \\
  @>\delta_2>>H_1(S;\rho)
  @>\varphi_1>>H_1(C;\rho)\oplus H_1(\pattern;\rho)
  @>\psi_1>>H_1(E;\rho)
  \\
  @>\delta_1>>H_0(S;\rho)
  @>\varphi_0>>H_0(C;\rho)\oplus H_0(\pattern;\rho)
  @>\psi_0>>H_0(E;\rho).
\end{CD}
\end{equation}
Here we put $\varphi_k:=i^C_{\ast}\oplus i^{\pattern}_{\ast}$ and $\psi_k:=j^C_{\ast}-j^{\pattern}_{\ast}$, and $\delta_i$ is the connecting homomorphism.
Note that the degrees of $\mathcal{H}_{\ast}$ is defined as
\begin{align*}
  \mathcal{H}_{3k}
  &=
  H_{k}(E;\rho),
  \\
  \mathcal{H}_{3k+1}
  &=
  H_{k}(C;\rho)\oplus H_{k}(D;\rho),
  \\
  \mathcal{H}_{3k+2}
  &=
  H_{k}(S;\rho).
\end{align*}
\begin{rem}
Lines~5--7, Page~8 in \cite{Porti:MAMCAU1997} should be read as
\begin{align*}
  \mathcal{H}_{3i}  
  &=
  H_{i}({\pattern}_{\ast}),
  \\
  \mathcal{H}_{3i+1}
  &=
  H_{i}(N_{\ast}),
  \\
  \mathcal{H}_{3i+2}
  &=
  H_{i}(M_{\ast}).
\end{align*}
See \cite[Footnote, P.~367]{Milnor:BULAM11966}.
\end{rem}
\subsection{Reidemeister torsion of a torus}\label{subsec:torus}
In this subsection we calculate the Reidemeister torsion of the torus $S$.
\par
In our case we assume that the meridian $\mu_{C}$ of $S$ is sent to a diagonal matrix $\begin{pmatrix}\zeta&0\\0&\zeta^{-1}\end{pmatrix}$ with $\zeta^2\ne1$ by conjugation.
Since the longitude $\lambda_{C}$ commutes with $\mu_{C}$, it is sent to $\begin{pmatrix}\eta&0\\0&\eta^{-1}\end{pmatrix}$ ($\eta^2$ may be 1).
\par
We regard $S$ as a cell complex with one $0$-cell $v$ (the basepoint), two $1$-cells $\mu_{C}$ and $\lambda_{C}$, and one $2$-cell $f$.
Here $f$ is attached to the bouquet $v\cup\mu_{C}\cup\lambda_{C}$ so that $\partial{f}$ is identified with $\mu_{C}\lambda_{C}(\mu_{C})^{-1}(\lambda_{C})^{-1}$.
So the $0$-dimensional chain complex $C_0(S;\rho)$ is spanned by $\{\tilde{v}\otimes E,\tilde{v}\otimes H,\tilde{v}\otimes F\}$, where $\tilde{v}$ is the lift of $v$ in the universal cover of $S$.
Similarly, $C_1(S;\rho)$ is spanned by $\{\tilde{\mu}_{C}\otimes E,\tilde{\mu}_{C}\otimes H,\tilde{\mu}_{C}\otimes F,\tilde{\lambda}_{C}\otimes E,\tilde{\lambda}_{C}\otimes H,\tilde{\lambda}_{C}\otimes F\}$, and $C_2(S;\rho)$ is spanned by $\{\tilde{f}\otimes E,\tilde{f}\otimes H,\tilde{f}\otimes F\}$.
\par
Since we calculate
\begin{align*}
  \rho(\mu_{C})^{-1}E\rho(\mu_{C})
  &=
  \begin{pmatrix}
    0&\zeta^{-2}\\
    0&0
  \end{pmatrix},
  \\
  \rho(\mu_{C})^{-1}H\rho(\mu_{C})
  &=
  \begin{pmatrix}
    1&0\\
    0&-1
  \end{pmatrix},
  \\
  \rho(\mu_{C})^{-1}F\rho(\mu_{C})
  &=
  \begin{pmatrix}
    0        &0\\
    \zeta^{2}&0
  \end{pmatrix},
\end{align*}
the adjoint action of $\rho(\mu_{C})$ is given by the multiplication from the left of the following three by three matrix:
\begin{equation*}
  M:=
  \begin{pmatrix}
    \zeta^{-2}&0&0 \\
    0         &1&0 \\
    0         &0&\zeta^2
  \end{pmatrix}.
\end{equation*}
Similarly, the adjoint action of $\rho(\lambda_{C})$ is given by
\begin{equation*}
  L:=
  \begin{pmatrix}
    \eta^{-2}&0&0 \\
    0        &1&0 \\
    0        &0&\eta^2
  \end{pmatrix}.
\end{equation*}
\par
The differential $\partial_2^{S}$ is given by
\begin{equation*}
\begin{split}
  \partial_2^{S}
  &=
  \Phi
  \begin{pmatrix}
    \Ad{\rho}
    \left(
      \frac{\partial\left(\mu_{C}\lambda_{C}(\mu_{S})^{-1}(\lambda_{C})^{-1}\right)}
           {\partial{\mu_{C}}}
    \right)
    \\
    \Ad{\rho}
    \left(
      \frac{\partial\left(\mu_{C}\lambda_{C}(\mu_{S})^{-1}(\lambda_{C})^{-1}\right)}
           {\partial{\lambda_{C}}}
    \right)
  \end{pmatrix}
  \\
  &=
  \Phi
  \begin{pmatrix}
    \Ad{\rho}
    \left(
      1-\mu_{C}\lambda_{C}(\mu_{C})^{-1}
    \right)
    \\
    \Ad{\rho}
    \left(
      \mu_{C}-\mu_{C}\lambda_{C}(\mu_{C})^{-1}(\lambda_{C})^{-1}
    \right)
  \end{pmatrix}
  \\
  &=
  \begin{pmatrix}
    I_3-M^{-1}LM
    \\
    M-L^{-1}M^{-1}LM
  \end{pmatrix}
  \\
  &=
  \begin{pmatrix}
    I_3-L \\
    M-I_3
  \end{pmatrix}
  \\
  &=
  \begin{pmatrix}
    1-\eta^{-2} &0&0          \\
    0           &0&0          \\
    0           &0&1-\eta^{2} \\
    \zeta^{-2}-1&0&0          \\
    0           &0&0          \\
    0           &0&\zeta^{2}-1
  \end{pmatrix},
\end{split}
\end{equation*}
where $I_3$ is the three by three identity matrix.
The differential $\partial_1^{S}$ is given by
\begin{equation*}
\begin{split}
  \partial_1^{S}
  &=
  \Phi
  \begin{pmatrix}
    \Ad{\rho}(\mu_{C}-1)&\Ad{\rho}(\lambda_{C}-1)
  \end{pmatrix}
  \\
  &=
  \begin{pmatrix}
    \zeta^{-2}-1&0&0          &\eta^{-2}-1&0&0         \\
    0           &0&0          &0          &0&0         \\
    0           &0&\zeta^{2}-1&0          &0&\eta^{2}-1
  \end{pmatrix}.
\end{split}
\end{equation*}
\par
Therefore we have
\begin{align*}
  \Ker\partial_2^{S}
  &=
  \left\langle
    \begin{pmatrix}
      0\\1\\0
    \end{pmatrix}
  \right\rangle,
  \\
  \Im\partial_2^{S}
  &=
  \left\langle
    \partial_2^{S}
      \begin{pmatrix}
        1\\0\\0
      \end{pmatrix},
    \partial_2^{S}
      \begin{pmatrix}
        0\\0\\1
      \end{pmatrix}
  \right\rangle,
  \\
  \Ker\partial_1^{S}
  &=
  \Im\partial_2^{S}
  \oplus
  \left\langle
    \begin{pmatrix}
      0\\1\\0\\0\\0\\0
    \end{pmatrix},
    \begin{pmatrix}
      0\\0\\0\\0\\1\\0
    \end{pmatrix}
  \right\rangle,
  \\
  \Im\partial_1^{S}
  &=
  \left\langle
    \begin{pmatrix}
      1\\0\\0
    \end{pmatrix},
    \begin{pmatrix}
      0\\0\\1
    \end{pmatrix}
  \right\rangle,
\end{align*}
and so
\begin{alignat*}{2}
  H_2(S;\rho)
  &=
  \C
  &\quad&\text{(generated by $[S]\otimes\begin{pmatrix}0\\1\\0\end{pmatrix}$)},
  \\
  H_1(S;\rho)
  &=
  \C^2
  &\quad&\text{(generated by 
  $\{\tilde{\mu}_{C}\otimes\begin{pmatrix}0\\1\\0\end{pmatrix},
      \tilde{\lambda}_{C}\otimes\begin{pmatrix}0\\1\\0\end{pmatrix}\}$)},
  \\
  H_0(S;\rho)
  &=
  \C
  &\quad&\text{(generated by $\tilde{v}\otimes\begin{pmatrix}0\\1\\0\end{pmatrix}$)}.
\end{alignat*}
Here $[S]\in H_2(S;\Z)$ is the fundamental class.
Note also that $\begin{pmatrix}0\\1\\0\end{pmatrix}$ is invariant under the adjoint action of any element of $\pi_1(S)$.
\par
Now we calculate the Reidemeister torsion of $S$.
From the calculation above.
If we put
\begin{align*}
  \tilde{\mathbf{h}}^{S}_2
  &:=
  \left\langle
    \begin{pmatrix}
      0\\1\\0
    \end{pmatrix}
  \right\rangle,
  \\
  \mathbf{b}^{S}_2
  &:=
  \left\langle
    \begin{pmatrix}1\\0\\0\end{pmatrix},
    \begin{pmatrix}0\\0\\1\end{pmatrix}
  \right\rangle,
  \\
  \tilde{\mathbf{h}}^{S}_1
  &:=
  \left\langle
    \begin{pmatrix}
      0\\1\\0\\0\\0\\0
    \end{pmatrix},
    \begin{pmatrix}
      0\\0\\0\\0\\1\\0
    \end{pmatrix}
  \right\rangle,
  \\
  \mathbf{b}^{S}_1
  &:=
  \left\langle
    \begin{pmatrix}1\\0\\0\\0\\0\\0\end{pmatrix},
    \begin{pmatrix}0\\0\\1\\0\\0\\0\end{pmatrix}
  \right\rangle,
  \\
  \tilde{\mathbf{h}}^{S}_0
  &:=
  \left\langle
    \begin{pmatrix}0\\1\\0\end{pmatrix}
  \right\rangle,
\end{align*}
we see that $\mathbf{b}^{S}_2\cup\tilde{\mathbf{h}}^{S}_2$ forms a basis of $C^{S}_2$, that $\mathbf{b}^{S}_1\cup\tilde{\mathbf{h}}^{S}_1\cup\partial^{S}_2(\mathbf{b}^{S}_2)$ forms a basis of $C^{S}_1$, and that $\tilde{\mathbf{h}}^{S}_0\cup\partial^{S}_1(\mathbf{b}^{S}_1)$ forms a basis of $C^{S}_0$.
Therefore the Reidemeister torsion of $S$ associated with these bases is
\begin{equation}\label{eq:torsion_S}
\begin{split}
  \Tor(S;\rho)
  &=\pm
  \frac{
  \begin{vmatrix}
    1&0&0&0&1-\eta^{-2} &0          \\
    0&0&1&0&0           &0          \\
    0&1&0&0&0           &1-\eta^{2} \\
    0&0&0&0&\zeta^{-2}-1&0          \\
    0&0&0&1&0           &0          \\
    0&0&0&0&0           &\zeta^{2}-1
  \end{vmatrix}}
  {
  \begin{vmatrix}
    1&0&0 \\
    0&0&1 \\
    0&1&0
  \end{vmatrix}
  \begin{vmatrix}
    0&\zeta^{-2}-1&0          \\
    1&0           &0          \\
    0&0           &\zeta^{2}-1
  \end{vmatrix}}
  \\
  &=
  \pm1.
\end{split}
\end{equation}
\par
Note that we use $[S]\otimes H$ as the generator of $H_2(S;\rho)$, $\{\tilde{\mu}_{C}\otimes H,\tilde{\lambda}_{C}\otimes H\}$ as the generator set of $H_1(S;\rho)$, and $\tilde{v}\otimes H$ as the generator of $H_0(S;\rho)$.
Here $H$ can be characterized as the vector that is invariant under the adjoint action of $\rho(\mu_{C})$ (and hence that of $\rho(\lambda_{C})$).
\par
From \eqref{eq:MV_torsion} we have
\begin{equation}\label{eq:MV_torsion2}
  \Tor(E;\rho)
  =
  \frac{\Tor(\pattern;\rho)}{\Tor(S;\rho)\Tor(\mathcal{H}_{\ast})}.
\end{equation}
\subsection{$\Im\rho_{C}$ is Abelian and $\Im\rho_{\pattern}$ is non-Abelian}\label{subsec:AN}
\par
Let $\rho$ be the representation defined in Subection~\ref{subsec:rep_AN}.
Then we have the following theorem.
\begin{thm}\label{thm:AN}
The Reidemeister torsion $\Tor(E;\rho)$ of $E$ twisted by the adjoint action of the representation $\rho$ associated with the meridian is given by
\begin{equation*}
  \Tor(E;\rho)
  =
  \pm
  \frac{(2b+1)(\omega_2^{2a+1}+\omega_2^{-2a-1})^2}{2(\omega_2^2-\omega_2^{-2})^2}.
\end{equation*}
\end{thm}
\begin{cor}
Let $\rho^{\rm{AN}}_{\xi,j}$ be the representation $\rho$ with $z=\exp(\xi/2)$ and $\omega_2=\exp\left(\frac{(2j+1)\pi\sqrt{-1}}{2b+1}\right)$.
Then we have
\begin{equation*}
  \Tor(E;\rho^{\rm{AN}}_{\xi,j})
  =
  \pm
  \frac{2b+1}{2}
  \frac{\cos^{2}\left(\frac{(2a+1)(2j+1)\pi}{2b+1}\right)}{\sin^2\left(\frac{2(2j+1)\pi}{2b+1}\right)}
  =
  \pm\tau_1(\xi;j)^{-2}.
\end{equation*}
\end{cor}
The rest of this subsection will be devoted to the proof of Theorem~\ref{thm:AN}.
Note that this may follow from \cite[Theorem~3.7]{Kirk/Livingston:TOPOL1999} but for later use we will give a proof from scratch.
\subsubsection{Reidemeister torsion of $S$}
From \eqref{eq:torsion_S} we know that $\Tor(S;\rho)=\pm1$.
From the calculation there the associated bases for $H_{i}(S;\rho)$ ($i=0,1,2$) are as follows.
\par
Since we calculate
\begin{align*}
  \begin{pmatrix}\omega_2&z^{-1}\\0&\omega_2^{-1}\end{pmatrix}^{-1}
  E
  \begin{pmatrix}\omega_2&z^{-1}\\0&\omega_2^{-1}\end{pmatrix}
  &=
  \begin{pmatrix}
    0&\omega_2^{-2}\\
    0&0
  \end{pmatrix},
  \\
  \begin{pmatrix}\omega_2&z^{-1}\\0&\omega_2^{-1}\end{pmatrix}^{-1}
  H
  \begin{pmatrix}\omega_2&z^{-1}\\0&\omega_2^{-1}\end{pmatrix}
  &=
  \begin{pmatrix}
    1&2\omega_2^{-1}z^{-1}\\
    0&-1
  \end{pmatrix},
  \\
  \begin{pmatrix}\omega_2&z^{-1}\\0&\omega_2^{-1}\end{pmatrix}^{-1}
  F
  \begin{pmatrix}\omega_2&z^{-1}\\0&\omega_2^{-1}\end{pmatrix}
  &=
  \begin{pmatrix}
    -\omega_2z^{-1}&-z^{-2}\\
    \omega_2^{2}&\omega_2z^{-1}
  \end{pmatrix},
  \\
  \rho(\theta_1)^{-1}E\rho(\theta_1)
  &=
  \begin{pmatrix}
    -z^{-1}+\omega_2^{-1}z      &1 \\
    -(\omega_2^{-1}z-z^{-1})^{2}&z^{-1}-\omega_2^{-1}z
  \end{pmatrix},
  \\
  \rho(\theta_1)^{-1}H\rho(\theta_1)
  &=
  \begin{pmatrix}
    1                     &0 \\
    2z^{-1}-2\omega_2^{-1}z&-1
  \end{pmatrix},
  \\
  \rho(\theta_1)^{-1}F\rho(\theta_1)
  &=
  \begin{pmatrix}
    0&0 \\
    1&0
  \end{pmatrix},
\end{align*}
the adjoint action of $\rho(\mu_{C})=\rho(x)$ is given by the multiplication from the left of the following three by three matrix:
\begin{equation*}
  X:=
  \Theta_1
  \begin{pmatrix}
    \omega_2^{-2}&2\omega_2^{-1}z^{-1}&-z^{-2} \\
    0            &1                   &-\omega_2z^{-1}\\
    0            &0                   &\omega_2^2
  \end{pmatrix}
  \Theta_1^{-1}
\end{equation*}
with
\begin{equation*}
  \Theta_1
  :=
  \begin{pmatrix}
    1                           &0                      &0 \\
    -z^{-1}+\omega_2^{-1}z      &1                      &0 \\
    -(\omega_2^{-1}z-z^{-1})^{2}&2z^{-1}-2\omega_2^{-1}z&1
  \end{pmatrix}.
\end{equation*}
Note that the adjoint action of $[\lambda_C]$ is trivial.
\par
If we put
\begin{equation*}
  U
  :=
  \Theta_1\begin{pmatrix}2\\(\omega_2-\omega_2^{-1})z\\0\end{pmatrix}
  =
  \begin{pmatrix}
    2 \\
    z(\omega_2+\omega_2^{-1})-2z^{-1} \\
    2(\omega_2+\omega_2^{-1}-z^2-z^{-2})
  \end{pmatrix},
\end{equation*}
then $U$ is invariant under $X$ and so it is invariant under the adjoint action of any element in $\pi_1(S)$.
Therefore we have
\begin{itemize}
\item
$H_2(S;\rho)=\C$ is generated by $\{[S]\otimes U\}$,
\item
$H_1(S;\rho)=\C^2$  is generated by $\{\tilde{\mu}_{C}\otimes U,\tilde{\lambda}_{C}\otimes U\}$, and
\item
$H_0(S;\rho)=\C$ is generated by $\{\tilde{v}\otimes U$\}.
\end{itemize}
\subsubsection{Reidemeister torsion of $C$}
Since $\Delta(T(2,2a+1)^{(2,2b+1)})=\dfrac{(t^{2a+1}-t^{-(2a+1)})(t^{1/2}-t^{-1/2})}{(t-t^{-1})(t^{(2a+1)/2}-t^{-(2a+1)/2})}$ (see for example \cite{Lickorish:1997}) and the eigenvalues of the image of $\rho(\mu_{C})$ are $\omega_2^{\pm1}\ne\pm1$, we have
\begin{equation*}
\begin{split}
  \Tor(C;\rho)
  &=
  \pm
  \frac{1}{\left(\omega_2-\omega_2^{-1}\right)^2}
  \left(
    \frac{\left(\omega_2^{2(2a+1)}-\omega_2^{-2(2a+1)}\right)\left(\omega_2-\omega_2^{-1}\right)}
         {\left(\omega_2^2-\omega_2^{-2}\right)\left(\omega_2^{2a+1}-\omega_2^{-(2a+1)}\right)}
  \right)^2
  \\
  &=
  \pm
  \frac{1}{\left(\omega_2-\omega_2^{-1}\right)^2}
  \left(\frac{\omega_2^{2a+1}+\omega_2^{-2a-1}}{\omega_2+\omega_2^{-1}}\right)^2
\end{split}
\end{equation*}
from Proposition~\ref{prop:Alexander}.
From its proof we also have
\begin{alignat*}{2}
  H_2(C;\rho)
  &=
  \{0\},
  &&
  \\
  H_1(C;\rho)
  &=
  \C
  &\quad&\text{(generated by $\tilde{x}\otimes U$)},
  \\
  H_0(C;\rho)
  &=
  \C
  &\quad&\text{(generated by $\tilde{v}\otimes U$)}
\end{alignat*}
since $U$ is invariant under the adjoint action of $\rho(\mu_{C})$.
\subsubsection{Reidemeister torsion of $\pattern$}\label{subsubsec:AN_D}
We regard $\pattern$ as a cell complex with one $0$-cell $v$, two $1$-cells $p$ and $t$, and one $2$-cell $f$.
Note that $f$ is attached to the bouquet $v\cup p \cup t$ so that its boundary is identified with $ptptp^{-1}t^{-1}p^{-1}t^{-1}$.
\par
The adjoint actions of $\rho(p)$ and $\rho(t)$ are given by
\begin{equation*}
  P
  :=
  \begin{pmatrix}
    z^{-2}&2z^{-1}&-1 \\
    0     &1      &-z \\
    0     &0      &z^2
  \end{pmatrix}
  =
  \Theta_1
  \begin{pmatrix}P_1&P_2&P_3\end{pmatrix}
  \Theta_1^{-1}
\end{equation*}
with
\begin{align*}
  P_1
  &:=
  \begin{pmatrix}
    \omega_2^{-2}z^2\\
    \omega_2^{-3}(\omega_2-1)(z^2-\omega_2)z\\
    -\omega_2^{-4}(\omega_2-1)^2(z^2-\omega_2)^2
  \end{pmatrix},
  \\
  P_2
  &:=
  \begin{pmatrix}
    2\omega_2^{-1}z \\
    \omega_2^{-2}(2(\omega_2-1)z^2-\omega_2(\omega_2-2))\\
    -2\omega_2^{-3}(\omega_2-1)(z^2-\omega_2)((\omega_2-1)z^2+\omega_2)z^{-1}
  \end{pmatrix},
  \\
  P_3
  &:=
  \begin{pmatrix}
    -1 \\
    -z-z^{-1}+\omega_2^{-1}z \\
    (z+z^{-1}-\omega_2^{-1}z)^2
  \end{pmatrix}
\end{align*}
and
\begin{equation*}
\begin{split}
  T
  &:=
  \Theta_1
  \begin{pmatrix}
    \omega_2^{-2b}&\frac{2\omega_2(1-\omega_2^{-2b})}{(\omega_2^{2}-1)z}
      &-\frac{\omega_2^{2-2b}(\omega_2^{2b}-1)^2}{(\omega_2^2-1)^2z^2}\\
    0&1&-\frac{\omega_2(\omega_2^{2b}-1)}{(\omega_2^2-1)z} \\
    0&0&\omega_2^{2b}
  \end{pmatrix}
  \Theta_1^{-1}
  \\
  &=
  \Theta_1
  \begin{pmatrix}
    -\omega_2&\frac{2\omega_2}{(\omega_2-1)z}&\frac{\omega_2}{(\omega_2-1)^2z^2} \\
    0        &1                            &\frac{1}{(\omega_2-1)z} \\
    0        &0                            &-\omega_2^{-1}
  \end{pmatrix}
  \Theta_1^{-1}
\end{split}
\end{equation*}
since $\omega_2^{2b+1}=-1$.
\par
The differential $\partial_2$ is given by the following six by three matrix.
\begin{equation*}
\begin{split}
  &
  \partial_2^{\pattern}
  \\
  =&
  \Phi
  \begin{pmatrix}
    \Ad{\rho}\left(\frac{\partial(ptptp^{-1}t^{-1}p^{-1}t^{-1})}{\partial p}\right)
    \\
    \Ad{\rho}\left(\frac{\partial(ptptp^{-1}t^{-1}p^{-1}t^{-1})}{\partial t}\right)
  \end{pmatrix}
  \\
  =&
  \Phi
  \begin{pmatrix}
    \Ad{\rho}(1+pt-ptptp^{-1}-ptptp^{-1}t^{-1}p^{-1})
    \\
    \Ad{\rho}(p+ptp-ptptp^{-1}t^{-1}-ptptp^{-1}t^{-1}p^{-1}t^{-1})
  \end{pmatrix}
  \\
  =&
  \begin{pmatrix}
    I_3+TP-TPT-T \\
    P+PTP-PT-I_3
  \end{pmatrix}
  \\
  =&
  \begin{pmatrix}\Theta_1&O_3\\O_3&\Theta_1\end{pmatrix}
  \\&
  \times
  \diag
  \bigl(
    \omega_2^{-1}((\omega_2-1)^2z^2-\omega_2)z^{-2},
    \omega_2^{-2}(\omega_2-1)((\omega_2-1)z^2+\omega_2)z^{-1},
  \\
  &\phantom{=\times\diag(}
    \omega_2^{-3}(\omega_2-1)^2(z^2-\omega_2),1,\omega_2^{-1}(\omega_2-1)z,
    \omega_2^{-2}(\omega_2-1)^2(z^2-\omega_2)\bigr)
  \\
  &\times
  \begin{pmatrix}
    -1& 2& 1\\
     1&-2&-1\\
     1&-2&-1\\
     1&-2&-1\\
     1&-2&-1\\
    -1& 2& 1\\
  \end{pmatrix}
  \\
  &\times
  \diag
  \left(
    \frac{(\omega_2+1)(z^2-\omega_2)}{\omega_2^2},
    \frac{z^2-\omega_2}{\omega_2(\omega_2-1)z},
    \frac{(\omega_2^2-\omega_2+1)z^2-\omega_2}{(\omega_2-1)^2z^2}
  \right)
  \Theta_1^{-1},
\end{split}
\end{equation*}
where $\diag(c_1,c_2,\dots,c_n)$ denotes the diagonal matrix with diagonal entries $c_1,c_2,\dots,c_n$.
The differential $\partial_1^{\pattern}$ is given by the following three by six matrix.
\begin{equation*}
\begin{split}
  &\partial_1^{\pattern}
  \\
  =&
  \Phi
  \begin{pmatrix}
    \Ad{\rho}(p-1)&\Ad{\rho}(t-1)
  \end{pmatrix}
  \\
  =&
  \Theta_1
  \\
  \times&
  \left(
  \begin{array}{c:c:}
    \omega_2^{-2}z^2-1 &2\omega_2^{-1}z \\
    \omega_2^{-3}(\omega_2-1)(z^2-\omega_2)z&\omega_2^{-2}(2(\omega_2-1)z^2-\omega_2(\omega_2-2))-1\\
    -\omega_2^{-4}(\omega_2-1)^2(z^2-\omega_2)^2&-2\omega_2^{-3}(\omega_2-1)(z^2-\omega_2)((\omega_2-1)z^2+\omega_2)z^{-1}
  \end{array}
  \right.
  \\
  &\hspace{36mm}
  \left.
  \begin{array}{:c:c:c:c}
    -1&-\omega_2-1&\frac{2\omega_2}{\omega_2z-z}&\frac{\omega_2}{(\omega_2-1)^2z^2} \\
    -z-z^{-1}+\omega_2^{-1}z&0        &0                            &\frac{1}{(\omega_2-1)z} \\
    (z+z^{-1}-\omega_2^{-1}z)^2-1&0        &0                            &-\omega_2^{-1}-1
  \end{array}
  \right)
  \\
  \times&
  \begin{pmatrix}\Theta_1^{-1}&O_3\\O_3&\Theta_1^{-1}\end{pmatrix}.
\end{split}
\end{equation*}
Therefore we see that $\partial_1^{\pattern}$ is surjective and
\begin{align*}
  \Ker\partial_2^{\pattern}
  &=
  \left\langle
    \Theta_1
    \begin{pmatrix}
      2\\(\omega_2-\omega_2^{-1})z\\0
    \end{pmatrix},
    \Theta_1
    \begin{pmatrix}
      2\\
      z+z^{-1}-2\omega_2^{-1}z\\
      2\omega_2^{-2}(\omega_2-1)(z^2-\omega_2)
    \end{pmatrix}
  \right\rangle,
  \\
  \Im\partial_2^{\pattern}
  &=
  \left\langle
    \begin{pmatrix}\Theta_1&O_3\\O_3&\Theta_1\end{pmatrix}
    \begin{pmatrix}
      -\omega_2^{-1}((\omega_2-1)^2z^2-\omega_2)z^{-2} \\
      \omega_2^{-2}(\omega_2-1)((\omega_2-1)z^2+\omega_2)z^{-1} \\
      \omega_2^{-3}(\omega_2-1)^2(z^2-\omega_2) \\
      1 \\
      \omega_2^{-1}(\omega_2-1)z \\
      -\omega_2^{-2}(\omega_2-1)^2(z^2-\omega_2)
    \end{pmatrix}
  \right\rangle,
  \\
  \Ker\partial_1^{\pattern}
  &=
  \Im\partial_2^{\pattern}
  \\
  &\,\oplus
  \left\langle
    \begin{pmatrix}\Theta_1&O_3\\O_3&\Theta_1\end{pmatrix}
    \begin{pmatrix}
      2 \\
      z+z^{-1}-2\omega_2^{-1}z\\
      2\omega_2^{-2}(\omega_2-1)(z^2-\omega_2) \\
      0 \\
      0 \\
      0
    \end{pmatrix},
    \begin{pmatrix}\Theta_1&O_3\\O_3&\Theta_1\end{pmatrix}
    \begin{pmatrix}
      0\\0\\0\\2\\(\omega_2-\omega_2^{-1})z\\0
    \end{pmatrix}
  \right\rangle.
\end{align*}
\par
Putting $U:=
    \Theta_1
    \begin{pmatrix}
      2\\(\omega_2-\omega_2^{-1})z\\0
    \end{pmatrix}$
as before and $V:=
    \Theta_1
    \begin{pmatrix}
      2 \\
      z+z^{-1}-2\omega_2^{-1}z\\
      2\omega_2^{-2}(\omega_2-1)(z^2-\omega_2) \\
    \end{pmatrix}
    =
    \begin{pmatrix}
      2 \\
      z-z^{-1} \\
      0
    \end{pmatrix}$,
we have
\begin{alignat*}{2}
  H_2(\pattern;\rho)
  &=
  \C^2
  &\quad&\text{(generated by $\tilde{f}\otimes U$ and $\tilde{f}\otimes V]$)},
  \\
  H_1(\pattern;\rho)
  &=
  \C^2
  &\quad&\text{(generated by $\tilde{p}\otimes V$ and $\tilde{t}\otimes U$)},
  \\
  H_0(\pattern;\rho)
  &=
  \{0\}.
  &&
\end{alignat*}
Note that the vector $U$ is invariant under the adjoint actions of $\rho(\mu_C)=\rho(x)$ and $\rho(\lambda_C)=I_2$, and that $V$ is invariant under the adjoint actions of $\rho(\mu)=\rho(p)$ and $\rho(\lambda)=\rho(\lambda_{C}x^{b}pq^{-b}\lambda_{C}x^{b}p^{-3b-1})$, where $\mu$ and $\lambda$ are the meridian and the longitude of $T(2,2a+1)^{(2,2b+1)}$.
The latter is because the adjoint action of $\rho(\lambda)$ is given by
\begin{equation*}
  \begin{pmatrix}
    z^{8b+4}&-\frac{2(z^{8b+4}-1)}{z-z^{-1}}&-\frac{(z^{4b+2}-z^{-4b-2})^2}{(z-z^{-1})^2} \\
    0       &1                              &\frac{1-z^{-8b-4}}{z-z^{-1}} \\
    0       &0                              &z^{-8b-4}
  \end{pmatrix}.
\end{equation*}
\par
We want to study topological interpretation of the generators $\tilde{f}\otimes U$ and $\tilde{f}\otimes V$.
To do that we calculate $i^{\pattern}_{\ast}([S]\otimes U)$ and $i_{\ast}([\Sigma]\otimes V)$, where $\Sigma$ is the boundary component of $\pattern$ other than $S$ and $i\colon\Sigma\to\pattern$ is the inclusion.
\par
As described in Subsection~\ref{subsec:torus}, $[S]\otimes U$ generates $H_2(S;\rho)$.
Here $[S]$ is presented by the $2$-cell whose boundary is attached to
\begin{equation*}
  \mu_{C}\lambda_{C}\mu_{C}^{-1}\lambda_{C}^{-1}
  =
  x\cdot r\cdot x^{-1}\cdot r^{-1}
  =
  ptptp^{-1}t^{-1}p^{-1}t^{-1}.
\end{equation*}
Since this coincides with the boundary of $f$, we have
\begin{equation*}
  i^{\pattern}_{\ast}([S]\otimes U)
  =
  \tilde{f}\otimes U.
\end{equation*}
\par
Next we calculate $i_{\ast}([\Sigma]\otimes V)$.
As in the previous case, $[\Sigma]\otimes V$ generates $H_2(\Sigma;\rho)$ and $[\Sigma]$ is presented by the $2$-cell whose boundary is attached to $\mu\lambda\mu^{-1}\lambda^{-1}$.
Since $\lambda=tptp^{-1}$ from Figure~\ref{fig:2_4}, we have
\begin{equation*}
\begin{split}
  \mu\lambda\mu^{-1}\lambda^{-1}
  &=
  p\cdot tptp^{-1}\cdot p^{-1}\cdot\left(tptp^{-1}\right)^{-1}
  \\
  &=
  ptptp^{-1}t^{-1}p^{-1}t^{-1},
\end{split}
\end{equation*}
which coincides with the boundary of $f$.
So we also have $i^{\pattern}_{\ast}([\Sigma]\otimes V)=\tilde{f}\otimes V$.
Therefore we conclude that $H_2(\pattern;\rho)$ is generated by $i^{\pattern}_{\ast}([S]\otimes U)$ and $i^{\pattern}_{\ast}([\Sigma]\otimes V)$.
\par
Now we calculate the Reidemeister torsion of $\pattern$.
\par
From the calculation above if we put
\begin{align*}
  \tilde{\mathbf{h}}^{\pattern}_2
  &:=
  \left\langle
    \Theta_1
    \begin{pmatrix}
      2\\(\omega_2-\omega_2^{-1})z\\0
    \end{pmatrix},
    \Theta_1
    \begin{pmatrix}
      2 \\
      z+z^{-1}-2\omega_2^{-1}z \\
      2\omega_2^{-2}(\omega_2-1)(z^2-\omega_2)
  \end{pmatrix}
  \right\rangle,
  \\
  \mathbf{b}^{\pattern}_2
  &:=
  \left\langle
    \Theta_1
    \begin{pmatrix}
      \frac{\omega_2^2}{(\omega_2+1)(z^2-\omega_2)}\\0\\0
    \end{pmatrix}
  \right\rangle,
  \\
  \tilde{\mathbf{h}}^{\pattern}_1
  &:=
  \left\langle
    \begin{pmatrix}\Theta_1&O_3\\O_3&\Theta_1\end{pmatrix}
    \begin{pmatrix}
      2 \\
      z+z^{-1}-2\omega_2^{-1}z \\
      2\omega_2^{-2}(\omega_2-1)(z^2-\omega_2)\\
      0 \\
      0 \\
      0
    \end{pmatrix},
    \begin{pmatrix}\Theta_1&O_3\\O_3&\Theta_1\end{pmatrix}
    \begin{pmatrix}
      0\\0\\0\\2\\(\omega_2-\omega_2^{-1})z\\0
    \end{pmatrix}
  \right\rangle,
  \\
  \mathbf{b}^{\pattern}_1
  &:=
  \left\langle
    \begin{pmatrix}
    \Theta_1&O_3\\
    O_3     &\Theta_1
    \end{pmatrix}
    \begin{pmatrix}0\\1\\0\\0\\0\\0\end{pmatrix},
    \begin{pmatrix}
    \Theta_1&O_3\\
    O_3     &\Theta_1
    \end{pmatrix}
    \begin{pmatrix}0\\0\\0\\0\\1\\0\end{pmatrix},
    \begin{pmatrix}
    \Theta_1&O_3\\
    O_3     &\Theta_1
    \end{pmatrix}
    \begin{pmatrix}0\\0\\0\\0\\0\\1\end{pmatrix}
  \right\rangle,
\end{align*}
we see that $\tilde{\mathbf{h}}^{\pattern}_2\cup\mathbf{b}^{\pattern}_2$ forms a basis of $C^{\pattern}_2$, that $\tilde{\mathbf{h}}^{\pattern}_1\cup\mathbf{b}^{\pattern}_1\cup\partial^{\pattern}_2(\mathbf{b}^{\pattern}_2)$ forms a basis of $C^{\pattern}_1$, and that $\partial^{\pattern}_1(\mathbf{b}^{\pattern}_1)$ forms a basis of $C^{\pattern}_0$.
Therefore the Reidemeister torsion of $\pattern$ associated with these bases is
\begin{equation*}
\begin{split}
  &\Tor(\pattern;\rho)
  \\
  =&
  \pm
  \det\begin{pmatrix}\Theta_1&O_3\\O_3&\Theta_1\end{pmatrix}
  \\
  &\times
  \begin{vmatrix}
   2                                       &0&0&0&0&-\omega_2^{-1}((\omega_2-1)^2z^2-\omega_2)z^{-2}         \\
   z+z^{-1}-2\omega_2^{-1}z                &0&1&0&0&\omega_2^{-2}(\omega_2-1)((\omega_2-1)z^2+\omega_2)z^{-1}\\
   2\omega_2^{-2}(\omega_2-1)(z^2-\omega_2)&0&0&0&0&\omega_2^{-3}(\omega_2-1)^2(z^2-\omega_2)                \\
   0                                       &2&0&0&0&1                                                        \\
   0&(\omega_2-\omega_2^{-1})z               &0&1&0&\omega_2^{-1}(\omega_2-1)z                               \\
   0                                       &0&0&0&1&-\omega_2^{-2}(\omega_2-1)^2(z^2-\omega_2)               
  \end{vmatrix}
  \\
  &\times
  \left(
  \det\Theta_1
  \begin{vmatrix}
    2                        &2                                       &\frac{\omega_2^2}{(\omega_2+1)(z^2-\omega_2)}\\
    (\omega_2-\omega_2^{-1})z&z+z^{-1}-2\omega_2^{-1}z                &0\\
    0                        &2\omega_2^{-2}(\omega_2-1)(z^2-\omega_2)&0
  \end{vmatrix}
  \right)^{-1}
  \\
  &\times
  \left(
  \det\Theta_1
  \begin{vmatrix}
    2\omega_2^{-1}z&\frac{2\omega_2}{\omega_2z-z}&\frac{\omega_2}{(\omega_2-1)^2z^2} \\
    \omega_2^{-2}(2(\omega_2-1)z^2-\omega_2(\omega_2-2))-1&0&\frac{1}{(\omega_2-1)z} \\
    -2\omega_2^{-3}(\omega_2-1)(z^2-\omega_2)((\omega_2-1)z^2+\omega_2)z^{-1}\hspace{-3mm}&0&-\omega_2^{-1}-1
  \end{vmatrix}
  \right)^{-1}
  \\
  =&
  \pm
  \frac{1}{2}.
\end{split}
\end{equation*}
\subsubsection{Reidemeister torsion of the Mayer--Vietoris sequence}
From \eqref{eq:MV_torsion2} and the calculation above, we have
\begin{equation}\label{eq:MV_torsion_AN}
  \Tor(E;\rho)
  =
  \frac{\Tor(C;\rho)\Tor(\pattern;\rho)}{\Tor(\mathcal{H}_{\ast})}
  =
  \pm
  \frac{(\omega_2^{2a+1}+\omega_2^{-2a-1})^2}{2(\omega_2^2-\omega_2^{-2})^2\Tor(\mathcal{H}_{\ast})}.
\end{equation}
\par
We calculate the torsion of the Mayer--Vietoris exact sequence with respect to the bases that we have calculated:
\begin{align*}
  H_2(S;\rho)
  &=
  \left\langle
    [S]\otimes U
  \right\rangle,
  \\
  H_1(S;\rho)
  &=
  \left\langle
    \tilde{\mu}_{C}\otimes U,\tilde{\lambda}_{C}\otimes U
  \right\rangle,
  \\
  H_0(S;\rho)
  &=
  \left\langle
    \tilde{v}\otimes U
  \right\rangle,
  \\
  H_1(C;\rho)
  &=
  \left\langle
    \tilde{x}\otimes U
  \right\rangle,
  \\
  H_0(C;\rho)
  &=
  \left\langle
    \tilde{v}\otimes U
  \right\rangle,
  \\
  H_2(\pattern;\rho)
  &=
  \left\langle
    i^{\pattern}_{\ast}([S]\otimes U), i^{\pattern}_{\ast}([\Sigma]\otimes V)
  \right\rangle,
  \\
  H_1(\pattern;\rho)
  &=
  \left\langle
    \tilde{p}\otimes V, \tilde{t}\otimes U
  \right\rangle.
\end{align*}
So the exact sequence becomes
\begin{equation*}
\begin{CD}
  @>\delta_3>>\C
  @>\varphi_2>>\C^2
  @>\psi_2>>H_2(E;\rho)
  \\
  @>\delta_2>>\C^2
  @>\varphi_1>>\C\oplus\C^2
  @>\psi_1>>H_1(E;\rho)
  \\
  @>\delta_1>>\C
  @>\varphi_0>>\C
  @>\psi_0>>H_0(E;\rho).
\end{CD}
\end{equation*}
\par
We will study the homomorphisms $\varphi_k$ ($k=0,1,2$).
\begin{itemize}
\item
$\varphi_2$:
We know that $i^{C}_{\ast}$ sends $[S]\otimes U\in H_2(S;\rho)$ to $i^{\pattern}_{\ast}([S]\otimes U)\in H_2(\pattern;\rho)$, that $[S]\otimes U$ generates $H_2(S;\rho)$, and that $i^{\pattern}_{\ast}([S]\otimes U)$ is one of the generators of $H_2(\pattern;\rho)$.
Therefore we conclude that $\varphi_2$ is injective.
\item
$\varphi_0$:
We can also see that $i^{C}_{\ast}$ sends the generator $\tilde{v}\otimes U$ of $H_0(S;\rho)$ to the generator $\tilde{v}\otimes U$ of $H_0(C;\rho)$, and so $\varphi_0$ is an isomorphism.
\item
$\varphi_1$:
Since $\mu_{C}=x$ and $\lambda_{C}=y(xy)^{2a}x^{-4a-1}$ from \S~\ref{sec:pi1}, we see that $i^{C}_{\ast}$ sends $\tilde{\mu}_{C}\otimes U$ to $\tilde{x}\otimes U\in H_1(C;\rho)$ and $\tilde{\lambda}_{C}\otimes U$ to
\begin{equation*}
\begin{split}
  &i^{C}_{\ast}(\tilde{\lambda}_{C}\otimes U)
  \\
  =&
  \left(\frac{\partial(y(xy)^{2a}x^{-4a-1})}{\partial\,x}\tilde{x}\right)\otimes U
  +
  \left(\frac{\partial(y(xy)^{2a}x^{-4a-1})}{\partial\,y}\tilde{y}\right)\otimes U
  \\
  =&
  \tilde{x}\otimes
  \left(
    \sum_{k=0}^{2a-1}(YX)^{k}Y
    -
    \left(\sum_{k=1}^{4a+1}X^{-k}\right)(YX)^{2a}Y
  \right)U
  +
  \tilde{y}\otimes
  \left(
    \sum_{k=0}^{2a}(XY)^{k}
  \right)U
  \\
  =&
  0\in H_1(C;\rho)
\end{split}
\end{equation*}
since $\tilde{x}=\tilde{y}$, $X=Y$, and $XU=U$.
\par
On the other hand $\mu_{C}=ptpt^{-1}$ and $\lambda_{C}=t(ptpt^{-1})^{-b}$ in $\pi_1(\pattern)$ from \S~\ref{sec:pi1}.
Therefore we conclude that $i^{\pattern}_{\ast}$ sends $\tilde{\mu}_{C}\otimes U$ to
\begin{equation*}
\begin{split}
  &i^{\pattern}_{\ast}(\tilde{\mu}_{C}\otimes U)
  \\
  =&
  \left(\frac{\partial(ptpt^{-1})}{\partial\,p}\tilde{p}\right)\otimes U
  +
  \left(\frac{\partial(ptpt^{-1})}{\partial\,t}\tilde{t}\right)\otimes U
  \\
  =&
  \tilde{p}\otimes(I_3+TP)U
  +
  \tilde{t}\otimes(P-T^{-1}PTP)U
  \\
  =&
  \begin{pmatrix}\Theta_1&O_3\\O_3&\Theta_1\end{pmatrix}
  \begin{pmatrix}
    -2\omega_2^{-1}z^{-2}(z^2+1)((\omega_2-1)^2z^2-\omega_2) \\
    2\omega_2^{-1}(\omega_2-1)z^{-1}(z^2+1)((\omega_2-1)z^2+\omega_2) \\
    2\omega_2^{-3}(\omega_2-1)^2(z^2+1)(z^2-\omega_2) \\
    2(z^2-1) \\
    2\omega_2^{-1}(\omega_2-1)z(z^2-\omega_2) \\
    -2\omega_2^{-2}(\omega_2-1)^2(z^2+1)(z^2-\omega_2)
  \end{pmatrix}
  \\
  =&
  2(z^2+1)
  \begin{pmatrix}\Theta_1&O_3\\O_3&\Theta_1\end{pmatrix}
  \begin{pmatrix}
    -\omega_2^{-1}((\omega_2-1)^2z^2-\omega_2)z^{-2} \\
    \omega_2^{-2}(\omega_2-1)((\omega_2-1)z^2+\omega_2)z^{-1} \\
    \omega_2^{-3}(\omega_2-1)^2(z^2-\omega_2) \\
    1 \\
    \omega_2^{-1}(\omega_2-1)z \\
    -\omega_2^{-2}(\omega_2-1)^2(z^2-\omega_2)
  \end{pmatrix}
  \\
  &
  -2
  \begin{pmatrix}\Theta_1&O_3\\O_3&\Theta_1\end{pmatrix}
  \begin{pmatrix}
    0\\0\\0\\2\\(\omega_2-\omega_2^{-1})z\\0
  \end{pmatrix}
  \\
  =&
  -2\tilde{t}\otimes U
  \in
  H_1(\pattern;\rho),
\end{split}
\end{equation*}
and $\tilde{\lambda}_{C}\otimes U$ to
\begin{equation*}
\begin{split}
  &i^{\pattern}_{\ast}(\tilde{\lambda}_{C}\otimes U)
  \\
  =&
  \left(\frac{\partial(t(ptpt^{-1})^{-b})}{\partial\,p}\tilde{p}\right)\otimes U
  +
  \left(\frac{\partial(t(ptpt^{-1})^{-b})}{\partial\,t}\tilde{t}\right)\otimes U
  \\
  =&
  \tilde{p}\otimes
  \left(
    -
    \sum_{k=1}^{b}(P^{-1}T^{-1}P^{-1}T)^{k}T
    -
    \sum_{k=0}^{b-1}P^{-1}T(P^{-1}T^{-1}P^{-1}T)^{k}T
  \right)U
  \\
  &\quad
  +
  \tilde{t}\otimes
  \left(
    I_3
    +
    \sum_{k=0}^{b-1}(P^{-1}T^{-1}P^{-1}T)^{k}T
    -
    \sum_{k=0}^{b-1}T^{-1}P^{-1}T(P^{-1}T^{-1}P^{-1}T)^{k}T
  \right)U
  \\
  =&
  \begin{pmatrix}\Theta_1&O_3\\O_3&\Theta_1\end{pmatrix}
  \begin{pmatrix}
    2b\omega_2^{-1}z^{-2}(z^2+1)((\omega_2-1)^2z^2-\omega_2) \\
    -2b\omega_2^{-2}z^{-1}(z^2+1)(\omega_2-1)((\omega_2-1)z^2+\omega_2) \\
    -2b\omega_2^{-3}(z^2+1)(z^2-\omega_2)(\omega_2-1)^2 \\
    2-2b(z^2-1) \\
    \omega_2^{-1}z(\omega_2-1)(-2bz^2+2b\omega+\omega+1) \\
    2b\omega_2^{-2}(z^2+1)(z^2-\omega_2)(\omega_2-1)^2
  \end{pmatrix}
  \\
  =&
  -2b(z^2+1)
  \begin{pmatrix}\Theta_1&O_3\\O_3&\Theta_1\end{pmatrix}
  \begin{pmatrix}
    -\omega_2^{-1}((\omega_2-1)^2z^2-\omega_2)z^{-2} \\
    \omega_2^{-2}(\omega_2-1)((\omega_2-1)z^2+\omega_2)z^{-1} \\
    \omega_2^{-3}(\omega_2-1)^2(z^2-\omega_2) \\
    1 \\
    \omega_2^{-1}(\omega_2-1)z \\
    -\omega_2^{-2}(\omega_2-1)^2(z^2-\omega_2)
  \end{pmatrix}
  \\
  &+
  (2b+1)
  \begin{pmatrix}\Theta_1&O_3\\O_3&\Theta_1\end{pmatrix}
  \begin{pmatrix}
    0\\0\\0\\2\\(\omega_2-\omega_2^{-1})z\\0
  \end{pmatrix}
  \\
  =&
  (2b+1)\tilde{t}\otimes U
  \in H_1(\pattern;\rho).
\end{split}
\end{equation*}
\par
So $\varphi_1$ is presented by the matrix
\begin{equation*}
  \begin{pmatrix}
     1&0 \\
     0&0  \\
    -2&2b+1
  \end{pmatrix}
\end{equation*}
with respect to the bases mentioned above.
It is clear that $\varphi_1$ is injective.
\end{itemize}
Therefore we have $H_0(E;\rho)=\{0\}$, $H_2(E;\rho)=\C$, which is generated by $(j^{\pattern}_{\ast}\circ i^{\pattern}_{\ast})([\Sigma]\otimes V)$, and $H_1(E;\rho)=\C$, which is generated by $j^{\pattern}_{\ast}(\tilde{p}\otimes V)=\tilde{\mu}\otimes V$.
\par
Now the bases for $\mathcal{H}$ are given as follows:
\begin{align*}
  \mathbf{c}^{\mathcal{H}}_8
  &=
  \left\langle
    [S]\otimes U
  \right\rangle,
  \\
  \mathbf{c}^{\mathcal{H}}_7
  &=
  \left\langle
    i^{\pattern}_{\ast}([S]\otimes U),
    i^{\pattern}_{\ast}([\Sigma]\otimes V)
  \right\rangle,
  \\
  \mathbf{c}^{\mathcal{H}}_6
  &=
  \left\langle
    (j^{\pattern}_{\ast}\circ i^{\pattern}_{\ast})([\Sigma]\otimes V)
  \right\rangle,
  \\
  \mathbf{c}^{\mathcal{H}}_5
  &=
  \left\langle
    \tilde{\mu}_{C}\otimes U,
    \tilde{\lambda}_{C}\otimes U
  \right\rangle,
  \\
  \mathbf{c}^{\mathcal{H}}_4
  &=
  \left\langle
    \tilde{x}\otimes U,
    \tilde{p}\otimes V,
    \tilde{t}\otimes U
  \right\rangle,
  \\
  \mathbf{c}^{\mathcal{H}}_3
  &=
  \left\langle
    \tilde{\mu}\otimes V
  \right\rangle,
  \\
  \mathbf{c}^{\mathcal{H}}_2
  &=
  \left\langle
    \tilde{v}\otimes U
  \right\rangle,
  \\
  \mathbf{c}^{\mathcal{H}}_1
  &=
  \left\langle
    \tilde{v}\otimes U
  \right\rangle.
\end{align*}
\par
We choose $\mathbf{b}^{\mathcal{H}}_i$ as follows:
\begin{align*}
  \\
  \mathbf{b}^{\mathcal{H}}_8
  &=
  \mathbf{c}^{\mathcal{H}}_8,
  \\
  \mathbf{b}^{\mathcal{H}}_7
  &=
  \left\langle
    i^{\pattern}_{\ast}([\Sigma]\otimes V)
  \right\rangle,
  \\
  \mathbf{b}^{\mathcal{H}}_5
  &=
  \mathbf{c}^{\mathcal{H}}_5,
  \\
  \mathbf{b}^{\mathcal{H}}_4
  &=
  \left\langle
    \tilde{p}\otimes V
  \right\rangle,
  \\
  \mathbf{b}^{\mathcal{H}}_2
  &=
  \mathbf{c}^{\mathcal{H}}_2.
\end{align*}
Then the torsion of the Mayer-Vietoris sequence equals
\begin{equation*}
\begin{split}
  \Tor(\mathcal{H}_{\ast})
  &=
  \pm
  \frac{
  \left[
    \partial(\mathbf{b}^{\mathcal{H}}_8)
    \cup
    \mathbf{b}^{\mathcal{H}}_7
    \mid
    \mathbf{c}^{\mathcal{H}}_7
  \right]
  \left[
    \mathbf{b}^{\mathcal{H}}_5
    \mid
    \mathbf{c}^{\mathcal{H}}_5
  \right]
  \left[
    \partial(\mathbf{b}^{\mathcal{H}}_4)
    \mid
    \mathbf{c}^{\mathcal{H}}_3
  \right]
  \left[
    \partial(\mathbf{b}^{\mathcal{H}}_2)
    \mid
    \mathbf{c}^{\mathcal{H}}_1
  \right]
  }
  {
  \left[
    \mathbf{b}^{\mathcal{H}}_8
    \mid
    \mathbf{c}^{\mathcal{H}}_8
  \right]
  \left[
    \partial(\mathbf{b}^{\mathcal{H}}_7)
    \mid
    \mathbf{c}^{\mathcal{H}}_6
  \right]
  \left[
    \partial(\mathbf{b}^{\mathcal{H}}_5)
    \cup
    \mathbf{b}^{\mathcal{H}}_4
    \mid
    \mathbf{c}^{\mathcal{H}}_4
  \right]
  \left[
    \mathbf{b}^{\mathcal{H}}_2
    \mid
    \mathbf{c}^{\mathcal{H}}_2
  \right]
  }
  \\
  &=
  \pm
  \frac{1}
  {
  \det
  \begin{pmatrix}
     1&0   &0 \\
     0&0   &1 \\
    -2&2b+1&0
  \end{pmatrix}
  }
  \\
  &=
  \pm\frac{1}{2b+1}.
\end{split}
\end{equation*}
\par
So we finally have
\begin{equation*}
  \Tor(E;\rho)
  =
  \pm
  \frac{(2b+1)(\omega_2^{2a+1}+\omega_2^{-2a-1})^2}{2(\omega_2^2-\omega_2^{-2})^2}
\end{equation*}
from \eqref{eq:MV_torsion_AN}.

\subsection{$\Im\rho_{C}$ is non-Abelian and $\Im\rho_{\pattern}$ is Abelian}\label{subsec:NA}
Let $\rho$ be the representation defined in Subsection~\ref{subsec:rep_NA}.
Then we have the following theorem.
\begin{thm}\label{thm:NA}
The Reidemeister torsion $\Tor(E;\rho)$ of $E$ twisted by the adjoint action of the representation $\rho$ associated with the meridian is given by
\begin{equation*}
  \Tor(E;\rho)
  =
  \pm
  \frac{2a+1}{2}
  \left(
    \frac{z^{8a-2b+3}+z^{-8a+2b-3}}{\omega_1-\omega_1^{-1}}
  \right)^2.
\end{equation*}
\end{thm}
\begin{cor}
Let $\rho^{\rm{NA}}_{\xi,k}$ be the representation $\rho$ with $z=\exp(\xi/2)$ and $\omega_1=\exp\left(\frac{(2k+1)\pi\sqrt{-1}}{2a+1}\right)$.
Then we have
\begin{equation*}
  \Tor(E;\rho^{\rm{NA}}_{\xi,k})
  =
  \pm
  \frac{2a+1}{2}
  \frac{\cosh^{2}\left(\frac{(8a-2b+3)\xi}{2}\right)}{\sin^2\left(\frac{(2k+1)\pi}{2a+1}\right)}
  =
  \pm\tau_2(\xi;k)^{-2}.
\end{equation*}
\end{cor}
The rest of this subsection will be devoted to the proof of Theorem~\ref{thm:NA}
\subsubsection{Reidemeister torsion of $S$}
Since we calculate
\begin{align*}
  \rho(x)^{-1}E\rho(x)
  &=
  \begin{pmatrix}
    0&z^{-4}\\
    0&0
  \end{pmatrix},
  \\
  \rho(x)^{-1}H\rho(x)
  &=
  \begin{pmatrix}
    1&2z^{-2}\\
    0&-1
  \end{pmatrix},
  \\
  \rho(x)^{-1}F\rho(x)
  &=
  \begin{pmatrix}
    -z^{2}&-1\\
    z^{4}  &z^{2}
  \end{pmatrix},
  \\
\end{align*}
the adjoint actions of $\rho(\mu_{C})$ is given by the multiplications from the left of the following three by three matrix:
\begin{equation*}
  X
  :=
  \begin{pmatrix}
    z^{-4}&2z^{-2}&-1 \\
    0     &1      &-z^{2}\\
    0     &0      &z^{4}
  \end{pmatrix}.
\end{equation*}
If we put
\begin{equation*}
  W
  :=
  \begin{pmatrix}2\\z^2-z^{-2}\\0\end{pmatrix},
\end{equation*}
then $W$ is invariant under the adjoint action of any element of $\pi_1(S)$.
\par
Therefore we have
\begin{alignat*}{2}
  H_2(S;\rho)
  &=
  \C
  &\quad&\text{(generated by $[S]\otimes W$)},
  \\
  H_1(S;\rho)
  &=
  \C^2
  &\quad&\text{(generated by $\{\tilde{\mu}_{C}\otimes W,\tilde{\lambda}_{C}\otimes W\}$)},
  \\
  H_0(S;\rho)
  &=
  \C
  &\quad&\text{(generated by $\tilde{v}\otimes W$)}
\end{alignat*}
from Subsection~\ref{subsec:torus}.
\subsubsection{Reidemeister torsion of $C$}
In \cite[6.2]{Dubois:CANMB2006}, J.~Dubois proved that the \emph{cohomological} Reidemeister torsion of $C$ twisted by $\rho$ associated with the \emph{longitude} is $\pm\left(\dfrac{\omega_1-\omega_1^{-1}}{2a+1}\right)^2$.
Since we are using \emph{homological} Reidemeister torsion associated with the \emph{meridian}, we need to take its inverse and multiply by $\dfrac{\partial{u}}{\partial{v}}$ from \cite[Th{\'e}or{\`e}me~4.1]{Porti:MAMCAU1997}, where if we put $z^2:=e^{u/2}$ in $\rho(\mu_{C})=\rho(x)$ then $\rho(\lambda_{C})$ is given as $\begin{pmatrix}e^{v/2}&\ast\\0&e^{-v/2}\end{pmatrix}$.
Since $e^{v/2}=-z^{-8a-4}$ from Subsection~\ref{subsec:rep_NA}, we have
\begin{equation}\label{eq:du/dv}
  \frac{\partial{u}}{\partial{v}}
  =
  \frac{1}{-4a-2}.
\end{equation}
Therefore we have
\begin{equation*}
  \Tor(C;\rho)
  =
  \pm
  \frac{(2a+1)^2}{(\omega_1-\omega_1^{-1})^2}\times\frac{1}{(-4a-2)}
  =
  \pm
  \frac{2a+1}{2(\omega_1-\omega_1^{-1})^2}.
\end{equation*}
However for later use we give a proof by direct calculation.
\par
We regard $C$ as a cell complex with one $0$-cell $v$, two $1$-cells $x$ and $y$, and one $2$-cell $f$, where the boundary of $f$ is attached to the bouquet $v\cup x\cup y$ along the word $(xy)^ax(xy)^{-a}y^{-1}$.
Since we have
\begin{align*}
  \rho(y)^{-1}E\rho(y)
  &=
  \begin{pmatrix}
    z^{-2}(\omega_1+\omega_1^{-1}-z^4-z^{-4})&z^{-4}\\
    -(\omega_1+\omega_1^{-1}-z^4-z^{-4})^2   &-z^{-2}(\omega_1+\omega_1^{-1}-z^4-z^{-4})
  \end{pmatrix},
  \\
  \rho(y)^{-1}H\rho(y)
  &=
  \begin{pmatrix}
    1                                           &0\\
    -2z^{2}(\omega_1+\omega_1^{-1}-z^{4}-z^{-4})&-1
  \end{pmatrix},
  \\
  \rho(y)^{-1}F\rho(y)
  &=
  \begin{pmatrix}
    0    &0\\
    z^{4}&0
  \end{pmatrix},
\end{align*}
the adjoint action of $\rho(y)$ is given by
\begin{equation*}
  Y
  :=
  \begin{pmatrix}
    z^{-4}                                   &0                                           &0\\
    z^{-2}(\omega_1+\omega_1^{-1}-z^4-z^{-4})&1                                           &0\\
    -(\omega_1+\omega_1^{-1}-z^4-z^{-4})^2   &-2z^{2}(\omega_1+\omega_1^{-1}-z^{4}-z^{-4})&z^{4}
  \end{pmatrix}.
\end{equation*}
\par
The differential $\partial_2^{C}\colon C_2(C;\rho)\to C_1(C;\rho)$ is given by
\begin{equation*}
\begin{split}
  \partial_2^{C}
  &=
  \Phi
  \begin{pmatrix}
    \Ad{\rho}
    \left(
      \frac{\partial\left((xy)^ax(xy)^{-a}y^{-1}\right)}{\partial x}
    \right)
    \\
    \Ad{\rho}
    \left(
      \frac{\partial\left((xy)^ax(xy)^{-a}y^{-1}\right)}{\partial y}
    \right)
  \end{pmatrix}
  \\
  &=
  \begin{pmatrix}
    \sum_{k=0}^{a}(YX)^{k}-\sum_{k=1}^{a}(YX)^{-k}X(YX)^{a}
    \\
    \sum_{k=0}^{a-1}X(YX)^{k}-\sum_{k=0}^{a}Y^{-1}(YX)^{-k}X(YX)^{a}
  \end{pmatrix}
  \\
  &=
  \frac{1}{(\omega_1^2-1)^2}
  \begin{pmatrix}
    D_1&D_2&D_3
  \end{pmatrix}
\end{split}
\end{equation*}
with
\begin{align}
  D_1&:=-z^{-4}(z^4-\omega_1)(z^4-\omega_1^{-1})D_3,
  \notag
  \\
  D_2
  &:=
  \begin{pmatrix}
    2\omega_1\left(\left(\omega_1^2+\omega_1+1\right)z^8-2\omega_1z^4-\omega_1-2a\left(z^4-\omega_1\right)
          \left(z^4 \omega_1-1\right)\right)z^{-6} \\[1mm]\hdashline
    (2a+1)\omega_1^4-2\left(z^4+a\left(z^8+3\right)+1\right)\omega_1^3z^{-4}\hfill \\
    +2\left(z^{12}+z^4+a\left(4 z^8+2\right)+1\right)\omega_1^2z^{-8} \\
    \hfill-2\left(z^4+a \left(z^8+3\right)+1\right)\omega_1z^{-4}+2a+1 \\[1mm]\hdashline
    2(2a+1)(\omega_1z^4-1)^2(z^4-\omega_1)^2z^{-10} \\[1mm]\hdashline
    -2(2a+1)\omega_1\left(z^4-\omega_1\right)\left(z^4 \omega_1-1\right)z^{-2} \\[1mm]\hdashline
    -2(2 a+1)\omega_1^2z^{8}
    +2\omega_1\left((\omega_1-1)\omega_1+3a\left(\omega_1^2+1\right)+1\right)z^4\hfill\\
    -\left((\omega_1-2) \omega_1^3-2\omega_1+2a\left(\omega_1^4+4 \omega_1^2+1\right)+1\right) \\
    \hfill+2\omega_1\left(a\omega_1^2-\omega_1+a\right)z^{-4} \\[1mm]\hdashline
    2\left(z^4-\omega_1\right)\left(z^4 \omega_1-1\right)\hfill \\
    \hfill
    \left(\left(z^8+2z^4-\omega_1-1\right)\omega_1+2a\left(z^4-\omega_1\right)\left(z^4\omega_1-1\right)-1\right)z^{-6}
  \end{pmatrix},
  \label{eq:D2}
  \\
  D_3
  &:=
  \begin{pmatrix}
    \omega_1\left(\left(\omega_1^2-2a\omega_1+\omega_1+1\right)z^4+2a\omega_1+\omega_1\right)z^{-4} \\
    -\left(z^4-1\right)
    \omega_1\left(a\left(z^4\left(\omega_1^2+1\right)-2\omega_1\right)-\left(z^4+1\right)\omega_1\right)z^{-6} \\
    (2a+1)\left(z^4-1\right)\left(z^4-\omega_1\right)\omega_1\left(z^4\omega_1-1\right)z^{-8} \\
    -(2a+1)\left(z^4-1\right)\omega_1^2 \\
    -\left(z^4-1\right)\omega_1\left(\left(z^4+1\right)\omega_1+a\left(2\omega_1z^4-\omega_1^2-1\right)\right)z^{-2} \\
    \left(z^4-\omega_1\right)\left(z^4 \omega_1-1\right)
    \left(\omega_1\left(z^4+2a\left(z^4-1\right)+\omega_1+1\right)+1\right)z^{-4}
  \end{pmatrix}
  \label{eq:D3}
\end{align}
and
\begin{equation*}
\begin{split}
  &\partial_1^{C}
  \\
  =&
  \Phi
  \begin{pmatrix}
    \Ad{\rho}(x-1)&\Ad{\rho}(y-1)
  \end{pmatrix}
  \\
  =&
  \begin{pmatrix}X-I_3&Y-I_3\end{pmatrix}
  \\
  =&
  \left(
  \begin{array}{c:c:c:}
    z^{-4}-1&2z^{-2}&-1     \\
    0       &0      &-z^{2} \\
    0       &0      &z^{4}-1
  \end{array}
  \right.
  \\
  &\hspace{18mm}
  \left.
  \begin{array}{:c:c:c}
    z^{-4}-1&0&0\\
     z^{-2}(\omega_1+\omega_1^{-1}-z^4-z^{-4})&0                                           &0\\
    -(\omega_1+\omega_1^{-1}-z^4-z^{-4})^2   &-2z^{2}(\omega_1+\omega_1^{-1}-z^{4}-z^{-4})&z^{4}-1
  \end{array}
  \right)
\end{split}
\end{equation*}
with respect to the bases $\{\tilde{v}\otimes E,\tilde{v}\otimes H,\tilde{v}\otimes F\}$ for $C_0(C;\rho)$, $\{\tilde{x}\otimes E,\tilde{x}\otimes H,\tilde{x}\otimes F,\tilde{y}\otimes E,\tilde{y}\otimes H,\tilde{y}\otimes F\}$ for $C_1(C;\rho)$, and $\{\tilde{f}\otimes E,\tilde{f}\otimes H,\tilde{f}\otimes F\}$ for $C_0(C;\rho)$.
\par
So we have
\begin{align*}
  \Ker\partial_2^{C}
  &=
  \left\langle
    \begin{pmatrix}
      z^4\\0\\(z^4-\omega_1)(z^4-\omega_1^{-1})
    \end{pmatrix}
  \right\rangle,
  \\
  \Im\partial_2^{C}
  &=
  \left\langle
    \partial_2^{C}
    \begin{pmatrix}0\\1\\0\end{pmatrix},
    \partial_2^{C}
    \begin{pmatrix}0\\0\\1\end{pmatrix}
  \right\rangle,
  \\
  \Ker\partial_1^{C}
  &=
  \Im\partial_2^{C}
  \oplus
  \left\langle
    \begin{pmatrix}
      2\\z^2-z^{-2}\\0\\0\\0\\0
    \end{pmatrix}
  \right\rangle
  \\
  &=
  \Im\partial_2^{C}
  \oplus
  \left\langle
    \tilde{x}\otimes W
  \right\rangle,
  \\
  \Im\partial_1^{C}
  &=
  \left\langle
    \partial_1^{C}
    \begin{pmatrix}
      0\\1\\0\\0\\0\\0
    \end{pmatrix},
    \partial_1^{C}
    \begin{pmatrix}
      0\\0\\1\\0\\0\\0
    \end{pmatrix},
    \partial_1^{C}
    \begin{pmatrix}
      0\\0\\0\\0\\0\\1
    \end{pmatrix}
  \right\rangle,
\end{align*}
where we put $W:=\begin{pmatrix}2\\z^2-z^{-2}\\0\end{pmatrix}$ as before.
\par
Now we study topological interpretation of the generator of $\Ker\partial_2^{C}$.
Let $[S]\in H_2(S;\Z)$ be the fundamental class.
We will calculate $i^{C}_{\ast}([S]\otimes W)$.
Since $[S]$ is presented by a $2$-cell whose boundary is attached to
\begin{equation*}
\begin{split}
  \mu_{C}\lambda_{C}\mu_{C}^{-1}\lambda_{C}^{-1}
  &=
  x\cdot y(xy)^{2a}x^{-4a-1}\cdot x^{-1}\cdot x^{4a+1}(xy)^{-2a}y^{-1}\cdot
  \\
  &=
  \left((xy)^{a}x(xy)^{-a}y^{-1}\right)\times\ad{y(xy)^{a}}\left(\left((xy)^{a}x(xy)^{-a}y^{-1}\right)^{-1}\right)
\end{split}
\end{equation*}
with $\ad{z}(w):=zwz^{-1}$, we have
\begin{equation*}
\begin{split}
  i^{C}_{\ast}([S]\otimes W)
  &=
  \tilde{f}\otimes W+
  \Ad{\rho(y(xy)^{a})}(-\tilde{f}\otimes W)
  \\
  &=
  (I_3-(YX)^{a}Y)\begin{pmatrix}2\\z^2-z^{-2}\\0\end{pmatrix}
  \\
  &=
  \begin{pmatrix}2\\0\\2(z^4+z^{-4}-\omega_1-\omega_1^{-1})\end{pmatrix}
  \\
  &=
  2z^{-4}\begin{pmatrix}z^4\\0\\(z^4-\omega_1)(z^4-\omega_1^{-1})\end{pmatrix}.
\end{split}
\end{equation*}
This is because if $w$ is the word presenting the boundary of $f$, then the element $\ad{z}(w)$ is lifted to $\tilde{z}\tilde{w}\tilde{z}^{-1}$ in the universal cover of $C$.
Since $\tilde{w}$ bounds a lift of $f$, $\tilde{z}\tilde{w}\tilde{z}^{-1}$ coincides with $z\cdot\tilde{w}$, where $\cdot$ means the action of $\pi_1(C)$ on the universal cover of $C$ by the covering transformation.
\par
Therefore we have
\begin{alignat*}{2}
  H_2(C;\rho)
  &=
  \C,
  &\quad&\text{(generated by $i^{C}_{\ast}([S]\otimes W$)},
  \\
  H_1(C;\rho)
  &=
  \C
  &\quad&\text{(generated by $\tilde{x}\otimes W$)},
  \\
  H_0(C;\rho)
  &=
  \{0\}.
  &
\end{alignat*}
\par
From the calculation above if we put
\begin{align*}
  \tilde{\mathbf{h}}^{C}_2
  &:=
  \left\langle
    \begin{pmatrix}
      2\\0\\2(z^4+z^{-4}-\omega_1-\omega_1^{-1})
    \end{pmatrix}
  \right\rangle,
  \\
  \mathbf{b}^{C}_2
  &:=
  \left\langle
    \begin{pmatrix}0\\1\\0\end{pmatrix},
    \begin{pmatrix}0\\0\\1\end{pmatrix}
  \right\rangle,
  \\
  \tilde{\mathbf{h}}^{C}_1
  &:=
  \left\langle
    \begin{pmatrix}
      2\\z^2-z^{-2}\\0\\0\\0\\0
    \end{pmatrix}
  \right\rangle,
  \\
  \mathbf{b}^{C}_1
  &:=
  \left\langle
    \begin{pmatrix}0\\1\\0\\0\\0\\0\end{pmatrix},
    \begin{pmatrix}0\\0\\1\\0\\0\\0\end{pmatrix},
    \begin{pmatrix}0\\0\\0\\0\\0\\1\end{pmatrix}
  \right\rangle,
\end{align*}
we see that $\mathbf{b}^{C}_2\cup\tilde{\mathbf{h}}^{C}_2$ forms a basis of $C_2(C;\rho)$, that $\mathbf{b}^{C}_1\cup\tilde{\mathbf{h}}^{C}_1\cup\partial^{C}_2(\mathbf{b}^{C}_2)$ forms a basis of $C_1(C;\rho)$, and that $\partial^{C}_1(\mathbf{b}^{C}_1)$ forms a basis of $C_0(C;\rho)$.
Therefore the Reidemeister torsion of $C$ associated with these bases is
\begin{equation*}
\begin{split}
  &\Tor(C;\rho)
  \\
  =&\pm
  \frac{
  \left|
  \begin{array}{c:c:c}
    \begin{array}{cccc}
      0&0&0&2         \\
      1&0&0&z^2-z^{-2}\\
      0&1&0&0         \\
      0&0&0&0         \\
      0&0&0&0         \\
      0&0&1&0
    \end{array}
    &\frac{1}{(\omega_1^2-1)^2}D_2&\frac{1}{(\omega_1^2-1)^2}D_3
  \end{array}
  \right|}
  {
  \begin{vmatrix}
    0&0&2                                    \\
    1&0&0                                    \\
    0&1&2(z^4+z^{-4}-\omega_1-\omega_1^{-4})
  \end{vmatrix}
  \begin{vmatrix}
    2z^{-2}&-1     &0      \\
    0      &-z^{2} &0      \\
    0      &z^{4}-1&z^{4}-1
  \end{vmatrix}}
  \\
  =&
  \pm
  \frac{2a+1}{2(\omega_1-\omega_1^{-1})^2}.
\end{split}
\end{equation*}
\subsubsection{Reidemeister torsion of $\pattern$}
As before, we regard $\pattern$ as a cell complex with one $0$-cell $v$, two $1$-cells $p$ and $t$, and one $2$-cell $f$.
The adjoint action of $\rho(p)$ is given by
\begin{equation*}
\begin{split}
  P
  :=
  \Theta_2
  \begin{pmatrix}
    z^{-2}&0&0 \\
    0     &1&0 \\
    0     &0&z^2
  \end{pmatrix}
  \Theta_2^{-1}
\end{split}
\end{equation*}
with
\begin{equation*}
  \Theta_2
  :=
  \begin{pmatrix}
    1&\frac{2z^2}{z^4-1}&-\frac{z^4}{(z^4-1)^2}\\
    0&1                 &\frac{z^2}{1-z^{4}} \\
    0&0                 &1
  \end{pmatrix}.
\end{equation*}
Since $\rho(t)=-\rho(p)^{-8a+2b-4}$, the adjoint action of $\rho(t)$ is given by $T:=P^{-8a+2b-4}$.
\par
The differential $\partial_2$ is given by
\begin{align*}
  &
  \partial_2^{\pattern}
  \\
  =&
  \Phi
  \begin{pmatrix}
    \Ad{\rho}\left(\frac{\partial(ptptp^{-1}t^{-1}p^{-1}t^{-1})}{\partial p}\right)
    \\
    \Ad{\rho}\left(\frac{\partial(ptptp^{-1}t^{-1}p^{-1}t^{-1})}{\partial t}\right)
  \end{pmatrix}
  \\
  =&
  \begin{pmatrix}
    I_3+TP-P^{-1}TPTP-P^{-1}T^{-1}P^{-1}TPTP \\
    P+PTP-T^{-1}P^{-1}TPTP-T^{-1}P^{-1}T^{-1}P^{-1}TPTP
  \end{pmatrix}
  \\
  =&
  \begin{pmatrix}
    (I_3-P^{-8a+2b-4})(I_3+P^{-8a+2b-3}) \\
    (P-I_3)(I_3+P^{-8a+2b-3})
  \end{pmatrix}
  \\
  =&
  \begin{pmatrix}\Theta_2&O_3\\O_3&\Theta_2\end{pmatrix}
  \\
  &\times
  \begin{pmatrix}
    (1-z^{16a-4b+8})(1+z^{16a-4b+6})&0&0 \\
    0&0&0 \\
    0&0&(1-z^{-16a+4b-8})(1+z^{-16a+4b-6}) \\
    (z^{-2}-1)(1+z^{16a-4b+6})&0&0 \\
    0&0&0 \\
    0&0&(z^{2}-1)(1+z^{-16a+4b-6})
  \end{pmatrix}
  \\
  &\times
  \Theta_2^{-1}.
\end{align*}
The differential $\partial_1^{\pattern}$ is given by
\begin{equation*}
\begin{split}
  &\partial_1^{\pattern}
  \\
  =&
  \Phi
  \begin{pmatrix}
    \Ad{\rho}(p-1)&\Ad{\rho}(t-1)
  \end{pmatrix}
  \\
  =&
  \Theta_2
  \begin{pmatrix}
    z^{-2}-1&0&0    &z^{16a-4b+8}-1&0&0 \\
    0       &0&0    &0             &0&0 \\
    0       &0&z^2-1&0             &0&z^{-16a+4b-8}-1
  \end{pmatrix}
  \begin{pmatrix}\Theta_2^{-1}&O_3\\O_3&\Theta_2^{-1}\end{pmatrix}.
\end{split}
\end{equation*}
Therefore we have
\begin{align*}
  \Ker\partial_2^{\pattern}
  &=
  \left\langle
    \Theta_2
    \begin{pmatrix}
      0\\1\\0
    \end{pmatrix}
  \right\rangle
  =
  \left\langle
    \frac{1}{z^2-z^{-2}}
    \begin{pmatrix}
      2\\z^2-z^{-2}\\0
    \end{pmatrix}
  \right\rangle,
  \\
  \Im\partial_2^{\pattern}
  &=
  \left\langle
    \begin{pmatrix}\Theta_2&O_3\\O_3&\Theta_2\end{pmatrix}
    \begin{pmatrix}
      1-z^{16a-4b+8} \\
      0\\
      0\\
      z^{-2}-1 \\
      0\\
      0
    \end{pmatrix},
    \begin{pmatrix}\Theta_2&O_3\\O_3&\Theta_2\end{pmatrix}
    \begin{pmatrix}
      0\\
      0\\
      1-z^{-16a+4b-8} \\
      0\\
      0\\
      z^{2}-1 \\
    \end{pmatrix}
  \right\rangle,
  \\
  \Ker\partial_1^{\pattern}
  &=
  \Im\partial_2^{\pattern}
  \oplus
  \left\langle
    \begin{pmatrix}\Theta_2&O_3\\O_3&\Theta_2\end{pmatrix}
    \begin{pmatrix}
      0\\1\\0\\0\\0\\0
    \end{pmatrix},
    \begin{pmatrix}\Theta_2&O_3\\O_3&\Theta_2\end{pmatrix}
    \begin{pmatrix}
      0\\0\\0\\0\\1\\0
    \end{pmatrix}
  \right\rangle
  \\
  &=
  \Im\partial_2^{\pattern}
  \oplus
  \left\langle
    \begin{pmatrix}2\\z^2-z^{-2}\\0\\0\\0\\0\end{pmatrix},
    \begin{pmatrix}0\\0\\0\\2\\z^2-z^{-2}\\0\end{pmatrix}
  \right\rangle,
  \\
  \Im\partial_1^{\pattern}
  &=
  \left\langle
    \begin{pmatrix}\Theta_2&O_3\\O_3&\Theta_2\end{pmatrix}
    \begin{pmatrix}z^{-2}-1\\0\\0\end{pmatrix},
    \begin{pmatrix}\Theta_2&O_3\\O_3&\Theta_2\end{pmatrix}
    \begin{pmatrix}0\\0\\z^2-1\end{pmatrix}.
  \right\rangle
\end{align*}
\par
So we have
\begin{alignat*}{2}
  H_2(\pattern;\rho)
  &=
  \C
  &\quad&\text{(generated by $\tilde{f}\otimes W$)},
  \\
  H_1(\pattern;\rho)
  &=
  \C^2
  &\quad&\text{(generated by $\tilde{p}\otimes W$ and $\tilde{t}\otimes W$)},
  \\
  H_0(\pattern;\rho)
  &=
  \C
  &\quad&\text{(generated by $\tilde{v}\otimes W$)}.
\end{alignat*}
\par
As in the previous case we have $i^{\pattern}_{\ast}([S]\otimes W)=\tilde{f}\otimes W$ and so $H_2(\pattern;\rho)$ is generated by $i^{\pattern}_{\ast}([S]\otimes W)$.
\begin{rem}\label{rem:NA_Sigma}
We calculate $i_{\ast}([\Sigma]\otimes W)$, where $i\colon\Sigma\to\pattern$ is the inclusion map.
As in Subsubsection~\ref{subsubsec:AN_D}, the fundamental group $[\Sigma]$ coincides with $f$.
Moreover it can be seen that the adjoint action of $\pi_1(\Sigma)$ leaves $W$ invariant.
So we see that $i_{\ast}([\Sigma]\otimes W)$ also coincides with $\tilde{f}\otimes W$.
\end{rem}
\par
Now we calculate the Reidemeister torsion of $\pattern$.
\par
From the calculation above if we put
\begin{align*}
  \tilde{\mathbf{h}}^{\pattern}_2
  &:=
  \left\langle
    \Theta_2
    \begin{pmatrix}
      0\\1\\0
    \end{pmatrix}
  \right\rangle,
  \\
  \mathbf{b}^{\pattern}_2
  &:=
  \left\langle
    \Theta_2
    \begin{pmatrix}
      1\\0\\0
    \end{pmatrix},
    \Theta_2
    \begin{pmatrix}
      0\\0\\1
    \end{pmatrix}
  \right\rangle,
  \\
  \tilde{\mathbf{h}}^{\pattern}_1
  &:=
  \left\langle
    \begin{pmatrix}\Theta_2&O_3\\O_3&\Theta_2\end{pmatrix}
    \begin{pmatrix}
      0\\1\\0\\0\\0\\0
    \end{pmatrix},
    \begin{pmatrix}\Theta_2&O_3\\O_3&\Theta_2\end{pmatrix}
    \begin{pmatrix}
      0\\0\\0\\0\\1\\0
    \end{pmatrix}
  \right\rangle,
  \\
  \mathbf{b}^{\pattern}_1
  &:=
  \left\langle
    \begin{pmatrix}
    \Theta_2&O_3\\
    O_3     &\Theta_2
    \end{pmatrix}
    \begin{pmatrix}1\\0\\0\\0\\0\\0\end{pmatrix},
    \begin{pmatrix}
    \Theta_2&O_3\\
    O_3     &\Theta_2
    \end{pmatrix}
    \begin{pmatrix}0\\0\\1\\0\\0\\0\end{pmatrix}
  \right\rangle,
  \\
  \tilde{\mathbf{h}}^{\pattern}_0
  &:=
  \left\langle
    \Theta_2
    \begin{pmatrix}
      0\\1\\0
    \end{pmatrix}
  \right\rangle,
\end{align*}
we see that $\tilde{\mathbf{h}}^{\pattern}_2\cup\mathbf{b}^{\pattern}_2$ forms a basis of $C^{\pattern}_2$, that $\tilde{\mathbf{h}}^{\pattern}_1\cup\mathbf{b}^{\pattern}_1\cup\partial^{\pattern}_2(\mathbf{b}^{\pattern}_2)$ forms a basis of $C^{\pattern}_1$, and that $\tilde{\mathbf{h}}^{\pattern}_0\cup\partial^{\pattern}_1(\mathbf{b}^{\pattern}_1)$ forms a basis of $C^{\pattern}_0$.
Therefore the Reidemeister torsion of $\pattern$ associated with these bases is
\begin{equation*}
\begin{split}
  &\Tor(\pattern;\rho)
  \\
  =&
  \pm
  \det\begin{pmatrix}\Theta_2&O_3\\O_3&\Theta_2\end{pmatrix}
  \\
  &\times
  \begin{vmatrix}
    0&0&1&0&\hspace{-1mm}(1-z^{16a-4b+8})(1+z^{16a-4b+6})\hspace{-2mm}&0\\
    1&0&0&0&0                               &0\\
    0&0&0&1&0                               &\hspace{-3mm}(1-z^{-16a+4b-8})(1+z^{-16a+4b-6})\\
    0&0&0&0&(z^{-2}-1)(1+z^{16a-4b+6})      &0\\
    0&1&0&0&0                               &0\\
    0&0&0&0&0                               &(z^{2}-1)(1+z^{-16a+4b-6})
  \end{vmatrix}
  \\
  &\times
  \left(
  \det\Theta_2
  \begin{vmatrix}
    0&1&0 \\
    1&0&0 \\
    0&0&1
  \end{vmatrix}
  \det\Theta_2
  \begin{vmatrix}
    0&z^{-2}-1&0 \\
    1&0       &0 \\
    0&0       &z^2-1
  \end{vmatrix}
  \right)^{-1}
  \\
  =&
  \pm
  (z^{8a-2b+3}+z^{-8a+2b-3})^2.
\end{split}
\end{equation*}
\subsubsection{Reidemeister torsion of the Mayer--Vietoris sequence}
From \eqref{eq:MV_torsion2} we have
\begin{equation}\label{eq:MV_torsion_NA}
  \Tor(E;\rho)
  =
  \frac{\Tor(C;\rho)\Tor(\pattern;\rho)}{\Tor(\mathcal{H}_{\ast})}
  =
  \pm
  \frac{(2a+1)(z^{8a-2b+3}+z^{-8a+2b-3})^2}{2(\omega_1-\omega_1^{-1})^2\Tor(\mathcal{H}_{\ast})},
\end{equation}
where $\mathcal{H}_{\ast}$ is the following Mayer--Vietoris exact sequence \eqref{eq:MV_exact}.
\par
We calculate the torsion of the Mayer--Vietoris exact sequence with respect to the bases that we have calculated:
\begin{align*}
  H_2(S;\rho)
  &=
  \left\langle
    [S]\otimes W
  \right\rangle,
  \\
  H_1(S;\rho)
  &=
  \left\langle
    \tilde{\mu}_{C}\otimes W,\tilde{\lambda}_{C}\otimes W
  \right\rangle,
  \\
  H_0(S;\rho)
  &=
  \left\langle
    \tilde{v}\otimes W
  \right\rangle,
  \\
  H_2(C;\rho)
  &=
  \left\langle
    i^{C}_{\ast}([S]\otimes W)
  \right\rangle,
  \\
  H_1(C;\rho)
  &=
  \left\langle
    \tilde{x}\otimes W
  \right\rangle,
  \\
  H_2(\pattern;\rho)
  &=
  \left\langle
    i^{\pattern}_{\ast}([S]\otimes W)
  \right\rangle,
  \\
  H_1(\pattern;\rho)
  &=
  \left\langle
    \tilde{p}\otimes W, \tilde{t}\otimes W
  \right\rangle,
  \\
  H_0(\pattern;\rho)
  &=
  \left\langle
    \tilde{v}\otimes W
  \right\rangle
\end{align*}
with $W:=\begin{pmatrix}2\\z^2-z^{-2}\\0\end{pmatrix}$.
So the exact sequence becomes
\begin{equation*}
\begin{CD}
  @>\delta_3>>\C
  @>\varphi_2>>\C\oplus\C
  @>\psi_2>>H_2(E;\rho)
  \\
  @>\delta_2>>\C^2
  @>\varphi_1>>\C\oplus\C^2
  @>\psi_1>>H_1(E;\rho)
  \\
  @>\delta_1>>\C
  @>\varphi_0>>\C
  @>\psi_0>>H_0(E;\rho).
\end{CD}
\end{equation*}
\par
We will study the homomorphisms $\varphi_k$ ($k=0,1,2$).
\begin{itemize}
\item
$\varphi_2$:
We know that $i^{C}_{\ast}$ sends $[S]\otimes W$ to $i^{C}_{\ast}([S]\otimes W)$.
Moreover from the calculation above, we see that $i^{\pattern}_{\ast}$ sends $[S]\otimes W$ to $i^{\pattern}_{\ast}([S]\otimes W)$.
Therefore $\varphi_2$ is injective.
\item
$\varphi_0$:
We can also see that $\varphi_0$ is the identity.
\item
$\varphi_1$:
Since $\mu_{C}=x$ and $\lambda_{C}=y(xy)^{2a}x^{-4a-1}$ from \S~\ref{sec:pi1}, we see that $i^{C}_{\ast}$ sends $\tilde{\mu}_{C}\otimes W$ to $\tilde{x}\otimes W\in H_1(C;\rho)$ and $\tilde{\lambda}_{C}\otimes W$ to
\begin{equation*}
\begin{split}
  &i^{C}_{\ast}(\tilde{\lambda}_{C}\otimes W)
  \\
  =&
  \left(\frac{\partial(y(xy)^{2a}x^{-4a-1})}{\partial\,x}\tilde{x}\right)\otimes W
  +
  \left(\frac{\partial(y(xy)^{2a}x^{-4a-1})}{\partial\,y}\tilde{y}\right)\otimes W
  \\
  =&
  \tilde{x}\otimes
  \left(
    \sum_{k=0}^{2a-1}(YX)^{k}Y
    -
    \left(\sum_{k=1}^{4a+1}X^{-k}\right)(YX)^{2a}Y
  \right)W
  +
  \tilde{y}\otimes
  \left(
    \sum_{k=0}^{2a}(XY)^{k}
  \right)W
  \\
  =&
  \frac{2a+1}{(\omega_1+1)^2}
  \begin{pmatrix}
    -2z^{-4}((2\omega_1^2+3\omega_1+2)z^4-\omega_1)
    \\
    z^{-6}(-(\omega_1^2+4\omega_1+1)z^8+(3\omega_1^2+2\omega_1+3)z^4-2\omega_1)
    \\
    -2z^{-8}(z^4+1)(z^4-\omega_1)(\omega_1z^4-1)
    \\
    2\omega_1(z^4+1)
    \\
    z^{-2}(z^4+1)(2\omega_1z^4-\omega_1^2-1)
    \\
    -2z^{-4}(z^4+1)(z^4-\omega_1)(\omega_1z^4-1)
  \end{pmatrix}
  \\
  =&
  -2(2a+1)\begin{pmatrix}2\\z^2-z^{-2}\\0\\0\\0\\0\end{pmatrix}
  +(z^2+z^{-2})\frac{D_2}{(\omega_1^2-1)^2}
  -2(z^4-z^{-4})\frac{D_3}{(\omega_1^2-1)^2}
  \\
  =&
  -2(2a+1)\tilde{x}\otimes W
  \in H_1(C;\rho),
\end{split}
\end{equation*}
where $\frac{D_2}{(\omega_1^2-1)^2}$ and $\frac{D_3}{(\omega_1^2-1)^2}$ are the second ant the third columns of $\partial^{C}_2$, respectively (see \eqref{eq:D2} and \eqref{eq:D3}).
Note that this is consistent with \eqref{eq:du/dv}.
\par
Since $\mu_{C}=ptpt^{-1}$ and $\lambda_{C}=t(ptpt^{-1})^{-b}$ in $\pi_1(\pattern)$ as before, $i^{\pattern}_{\ast}$ sends $\tilde{\mu}_{C}\otimes W$ to
\begin{equation*}
\begin{split}
  &i^{\pattern}_{\ast}(\tilde{\mu}_{C}\otimes W)
  \\
  =&
  \left(\frac{\partial(ptpt^{-1})}{\partial\,p}\tilde{p}\right)\otimes W
  +
  \left(\frac{\partial(ptpt^{-1})}{\partial\,t}\tilde{t}\right)\otimes W
  \\
  =&
  \tilde{p}\otimes(I_3+TP)W
  +
  \tilde{t}\otimes(P-T^{-1}PTP)W
  \\
  =&
  \begin{pmatrix}
    4\\2(z^2-z^{-2})\\0\\0\\0\\0
  \end{pmatrix}
  \\
  =&
  2\tilde{p}\otimes W
  \in
  H_1(\pattern;\rho),
\end{split}
\end{equation*}
and $\tilde{\lambda}_{C}\otimes W$ to
\begin{equation*}
\begin{split}
  &i^{\pattern}_{\ast}(\tilde{\lambda}_{C}\otimes W)
  \\
  =&
  \left(\frac{\partial(t(ptpt^{-1})^{-b})}{\partial\,p}\tilde{p}\right)\otimes W
  +
  \left(\frac{\partial(t(ptpt^{-1})^{-b})}{\partial\,t}\tilde{t}\right)\otimes W
  \\
  =&
  \tilde{p}\otimes
  \left(
    -
    \sum_{k=1}^{b}(P^{-1}T^{-1}P^{-1}T)^{k}T
    -
    \sum_{k=0}^{b-1}P^{-1}T(P^{-1}T^{-1}P^{-1}T)^{k}T
  \right)W
  \\
  &\quad
  +
  \tilde{t}\otimes
  \left(
    I_3
    +
    \sum_{k=0}^{b-1}(P^{-1}T^{-1}P^{-1}T)^{k}T
    -
    \sum_{k=0}^{b-1}T^{-1}P^{-1}T(P^{-1}T^{-1}P^{-1}T)^{k}T
  \right)W
  \\
  =&
  \begin{pmatrix}
    -4b\\-2b(z^2-z^{-2})\\0\\2\\z^2-z^{-2}\\0
  \end{pmatrix}
  \\
  =&
  -2b\tilde{p}\otimes W+\tilde{t}\otimes W
  \in H_1(\pattern;\rho).
\end{split}
\end{equation*}
\par
So $\varphi_1$ is presented by the matrix
\begin{equation*}
  \begin{pmatrix}
     1&-2(2a+1)       \\
     2&-2b\\
     0&1
  \end{pmatrix}
\end{equation*}
with respect to the bases mentioned above.
It is clear that $\varphi_1$ is injective.
\end{itemize}
Therefore we have $H_0(E;\rho)=\{0\}$, $H_2(E;\rho)=\C$, which is generated by $(j^{\pattern}_{\ast}\circ i^{\pattern}_{\ast})([S]\otimes W)$, and $H_1(E;\rho)=\C$, which is generated by $j^{\pattern}_{\ast}(\tilde{p})\otimes W$.
From Remark~\ref{rem:NA_Sigma}, $H_2(E;\rho)$ is also generated by $(j^{\pattern}_{\ast}\circ i^{\pattern}_{\ast})([\Sigma]\otimes W)$
\par
Now the bases for $\mathcal{H}$ are given as follows:
\begin{align*}
  \mathbf{c}^{\mathcal{H}}_8
  &=
  \left\langle
    [S]\otimes W
  \right\rangle,
  \\
  \mathbf{c}^{\mathcal{H}}_7
  &=
  \left\langle
    i^{C}_{\ast}([S]\otimes W),
    i^{\pattern}_{\ast}([S]\otimes W)
  \right\rangle,
  \\
  \mathbf{c}^{\mathcal{H}}_6
  &=
  \left\langle
    (j^{\pattern}_{\ast}\circ i^{\pattern}_{\ast})([S]\otimes W)
  \right\rangle,
  \\
  \mathbf{c}^{\mathcal{H}}_5
  &=
  \left\langle
    \tilde{\mu}_{C}\otimes W,
    \tilde{\lambda}_{C}\otimes W
  \right\rangle,
  \\
  \mathbf{c}^{\mathcal{H}}_4
  &=
  \left\langle
    \tilde{x}\otimes W,
    \tilde{p}\otimes W,
    \tilde{t}\otimes W
  \right\rangle,
  \\
  \mathbf{c}^{\mathcal{H}}_3
  &=
  \left\langle
    j^{\pattern}_{\ast}(\tilde{p}\otimes W)
  \right\rangle,
  \\
  \mathbf{c}^{\mathcal{H}}_2
  &=
  \left\langle
    \tilde{v}\otimes W
  \right\rangle,
  \\
  \mathbf{c}^{\mathcal{H}}_1
  &=
  \left\langle
    \tilde{v}\otimes W
  \right\rangle.
\end{align*}
\par
We choose $\mathbf{b}^{\mathcal{H}}_i$ as follows:
\begin{align*}
  \\
  \mathbf{b}^{\mathcal{H}}_8
  &=
  \mathbf{c}^{\mathcal{H}}_8,
  \\
  \mathbf{b}^{\mathcal{H}}_7
  &=
  \left\langle
    i^{\pattern}_{\ast}([S]\otimes W)
  \right\rangle,
  \\
  \mathbf{b}^{\mathcal{H}}_5
  &=
  \mathbf{c}^{\mathcal{H}}_5,
  \\
  \mathbf{b}^{\mathcal{H}}_4
  &=
  \left\langle
    \tilde{p}\otimes W
  \right\rangle,
  \\
  \mathbf{b}^{\mathcal{H}}_2
  &=
  \mathbf{c}^{\mathcal{H}}_2.
\end{align*}
Then the torsion of the Mayer-Vietoris sequence equals
\begin{equation*}
\begin{split}
  \Tor(\mathcal{H}_{\ast})
  &=
  \pm
  \frac{
  \left[
    \partial(\mathbf{b}^{\mathcal{H}}_8)
    \cup
    \mathbf{b}^{\mathcal{H}}_7
    \mid
    \mathbf{c}^{\mathcal{H}}_7
  \right]
  \left[
    \mathbf{b}^{\mathcal{H}}_5
    \mid
    \mathbf{c}^{\mathcal{H}}_5
  \right]
  \left[
    \partial(\mathbf{b}^{\mathcal{H}}_4)
    \mid
    \mathbf{c}^{\mathcal{H}}_3
  \right]
  \left[
    \partial(\mathbf{b}^{\mathcal{H}}_2)
    \mid
    \mathbf{c}^{\mathcal{H}}_1
  \right]
  }
  {
  \left[
    \mathbf{b}^{\mathcal{H}}_8
    \mid
    \mathbf{c}^{\mathcal{H}}_8
  \right]
  \left[
    \partial(\mathbf{b}^{\mathcal{H}}_7)
    \mid
    \mathbf{c}^{\mathcal{H}}_6
  \right]
  \left[
    \partial(\mathbf{b}^{\mathcal{H}}_5)
    \cup
    \mathbf{b}^{\mathcal{H}}_4
    \mid
    \mathbf{c}^{\mathcal{H}}_4
  \right]
  \left[
    \mathbf{b}^{\mathcal{H}}_2
    \mid
    \mathbf{c}^{\mathcal{H}}_2
  \right]
  }
  \\
  &=
  \pm
  \frac{1}
  {
  \left[
    \partial(\mathbf{b}^{\mathcal{H}}_5)
    \cup
    \mathbf{b}^{\mathcal{H}}_4
    \mid
    \mathbf{c}^{\mathcal{H}}_4
  \right]
  }
  \\
  &=
  \pm
  \frac{1}
  {
  \det
  \begin{pmatrix}
     1&-2(2a+1)&0 \\
     2&-2b     &1 \\
     0&1       &0
  \end{pmatrix}
  }
  \\
  &=
  \pm1.
\end{split}
\end{equation*}
\par
So we finally have
\begin{equation*}
  \Tor(E;\rho)
  =
  \pm
  \frac{2a+1}{2}
  \left(
    \frac{z^{8a-2b+3}+z^{-8a+2b-3}}{\omega_1-\omega_1^{-1}}
  \right)^2
\end{equation*}
from \eqref{eq:MV_torsion_NA}.
\begin{rem}
Since $\rho_{\pattern}$ is Abelian, $W$ is also invariant under the adjoint action of any element in $\pi_1(\pattern)$; especially it is invariant under the adjoint action of any element in $\pi_1(\Sigma)$.
\end{rem}
\subsection{Both $\Im\rho_{C}$ and $\Im\rho_{\pattern}$ are non-Abelian}
Let $\rho$ be the representation defined in Subsection~\ref{subsec:rep_NN}.
Then we have the following theorem.
\begin{thm}\label{thm:NN}
Let $\gamma$ be the element in $H_1(E;\rho)$ such that $\delta_1(\gamma)=\tilde{v}\otimes\tilde{U}\in H_0(E;\rho)$, where $\delta_1$ is the connecting homomorphism of the Mayer--Vietoris exact sequence associated with the triple $(C,D,C\cap D=S)$, $v$ is the basepoint in $S$, $\tilde{v}$ is its lift in the universal cover, and $\tilde{U}\in\sl(2;\C)$ is invariant under the adjoint action of any element in $H_1(S)$.
Then the Reidemeister torsion $\Tor(E;\rho)$ of $E$ twisted by the adjoint action of the representation $\rho$ associated with the meridian and $\gamma$ is given by
\begin{equation*}
  \Tor(E;\rho)
  =
  \pm
  \frac{(2a+1)(4(2a+1)-(2b+1))}{4(\omega_1-\omega_1^{-1})^2}.
\end{equation*}
\end{thm}
\begin{cor}
Let $\rho^{\rm{NN}}_{\xi,l,m}$ be the representation $\rho$ with $z=\exp(\xi/2)$ and $\omega_1=\exp\left(\frac{(2l+1)\pi\sqrt{-1}}{2a+1}\right)$ and $\omega_3=\exp\left(\frac{(2m+1)\pi\sqrt{-1}}{2b+1-4(2a+1)}\right)$.
Then we have
\begin{equation*}
  \Tor(E;\rho^{\rm{NN}}_{\xi,l,m})
  =
  \pm
  \frac{(2a+1)(2b+1-4(2a+1))}{16\sin^2\left(\frac{(2l+1)\pi}{2a+1}\right)}
  =
  \pm\tau_2(\xi;k)^{-2}.
\end{equation*}
\end{cor}
The rest of this subsection will be devoted to the proof of Theorem~\ref{thm:NN}
\subsubsection{Reidemeister torsion of $S$}
As before, we regard $S$ as a cell complex with one $0$-cell $v$ (the basepoint), two $1$-cells $\mu_{C}$ and $\lambda_{C}$, and one $2$-cell $f$.
\par
From \S\ref{subsec:rep_NN}, we have
\begin{align*}
  \rho(\mu_{C})
  &=
  \rho(x)
  =
  \theta_3^{-1}
  \begin{pmatrix}
    \omega_3&1\\
    0       &\omega_3^{-1}
  \end{pmatrix}
  \theta_3
  \\
  \rho(\lambda_{C})
  &=
  \theta_3^{-1}
  \begin{pmatrix}
    -\omega_3^{-4a-2}&\frac{\omega_3^{4a+2}-\omega_3^{-4a-2}}{\omega_3-\omega_3^{-1}} \\
    0                &-\omega_3^{4a+2}
  \end{pmatrix}
  \theta_3^{-1}.
\end{align*}
\par
Since these are obtained from those in Subsection~\ref{subsec:rep_NA} by replacing $z^2$ with $\omega_3$ and taking conjugation by $\theta_3$, the Reidemeister torsion is
\begin{equation*}
  \Tor(S;\rho)
  =
  \pm1.
\end{equation*}
\par
We give generators of the twisted homologies.
From the calculation in Subsection~\ref{subsec:NA}, we see that
\begin{itemize}
\item
$H_2(S;\rho)=\C$ is generated by $[S]\otimes\tilde{W}$,
\item
$H_1(S;\rho)=\C^2$ is generated by $\{\tilde{\mu}_C\otimes\tilde{W},\tilde{\lambda}_C\otimes\tilde{W}\}$, and
\item
$H_0(S;\rho)=\C$ is generated by $\tilde{v}\otimes\tilde{W}$.
\end{itemize}
where  $\tilde{W}:=\Theta_3\begin{pmatrix}2\\\omega_3-\omega_3^{-1}\\0\end{pmatrix}$ with
\begin{equation*}
  \Theta_3
  :=
  \begin{pmatrix}
    z^{-1}                              &0                      &0 \\
    \omega_3^{-1}-z^{-2}                &                       1&0 \\
    -\omega_3^{-2}z^{-3}(z^2-\omega_3)^2&2(z^{-1}-\omega_3^{-1}z)&z
  \end{pmatrix}.
\end{equation*}
Since
$\Theta_3=\tilde{\Theta}_1\begin{pmatrix}
    z^{-1}&0&0\\
    0     &1&0\\
    0     &0&z
  \end{pmatrix}$,
we have
\begin{equation*}
  \tilde{W}
  =
  \tilde{\Theta}_1
  \begin{pmatrix}
    2z^{-1}\\\omega_3-\omega_3^{-1}\\0
  \end{pmatrix}
  =
  z^{-1}\tilde{U}
\end{equation*}
with
\begin{equation*}
  \tilde{U}
  :=
  \tilde{\Theta}_1
  \begin{pmatrix}
    2\\(\omega_3-\omega_3^{-1})z\\0
  \end{pmatrix}
\end{equation*}
and
\begin{equation*}
  \tilde{\Theta}_1
  :=
  \begin{pmatrix}
    1                         &0                      &0\\
    -z^{-1}+\omega_3^{-1}z    &1                      &0 \\
    -(\omega_3^{-1}z-z^{-1})^2&2z^{-1}-2\omega_3^{-1}z&1
  \end{pmatrix}.
\end{equation*}
Note that $\tilde{U}$ ($\tilde{\Theta}_1$, respectively) is obtained from $U$ ($\Theta_1$, respectively) by replacing $\omega_2$ with $\omega_3$.
\par
So we can choose generators as follows without changing the Reidemeister torsion:
\begin{itemize}
\item
$H_2(S;\rho)=\C$ is generated by $[S]\otimes\tilde{U}$,
\item
$H_1(S;\rho)=\C^2$ is generated by $\{\tilde{\mu}_C\otimes\tilde{U},\tilde{\lambda}_C\otimes\tilde{U}\}$, and
\item
$H_0(S;\rho)=\C$ is generated by $\tilde{v}\otimes\tilde{U}$.
\end{itemize}
\subsubsection{Reidemeister torsion of $C$}
The matrices $\rho(x)$ and $\rho(y)$ are given from those in Subsection~\ref{subsec:rep_NA} by replacing $z^2$ with $\omega_3$ and taking conjugation by $\theta_3$.
Since the Reidemeister torsion of $C$ in Subsection~\ref{subsec:NA} does not involve $z$, the Reidemeister torsion in this case is the same and we have
\begin{equation*}
  \Tor(C;\rho)
  =
  \pm\frac{2a+1}{2(\omega_1-\omega_1^{-1})^2}.
\end{equation*}
Generators of the twisted homologies are given as follows:
\begin{itemize}
\item
$H_2(C;\rho)=\C$ is generated by $i^{C}_{\ast}([S]\otimes\tilde{U})$, and
\item
$H_1(C;\rho)=\C$ is generated by $\tilde{x}\otimes\tilde{U}$.
\end{itemize}
\subsubsection{Reidemeister torsion of $\pattern$}
The matrix $\rho(p)$ is the same as that in Subsection~\ref{subsec:AN}.
The adjoint action of $\rho(t)$ is
\begin{equation*}
\begin{split}
  &\tilde{\Theta}_1
  \begin{pmatrix}
    \omega_3^{-8a+2b-2}&\frac{2\omega_3(1-\omega_3^{-8a+2b-2})}{(\omega_3^{2}-1)z}
      &-\frac{\omega_3^{-8a+2b}(\omega_3^{8a-2b+2}-1)^2}{(\omega_3^2-1)^2z^2}\\
    0&1&-\frac{\omega_3(\omega_3^{8a-2b+2}-1)}{(\omega_2^2-1)z} \\
    0&0&\omega_2^{8a-2b+2}
  \end{pmatrix}
  \tilde{\Theta}_1^{-1}
  \\
  =&
  \tilde{\Theta}_1
  \begin{pmatrix}
    -\omega_3&\frac{2\omega_3}{(\omega_3-1)z}&\frac{\omega_3}{(\omega_3-1)^2z^2} \\
    0        &1                            &\frac{1}{(\omega_3-1)z} \\
    0        &0                            &-\omega_3^{-1}
  \end{pmatrix}
  \tilde{\Theta}_1^{-1}
\end{split}
\end{equation*}
since $\omega_3^{8a-2b+3}=-1$.
Therefore the adjoint actions of $\rho(p)$ and $\rho(t)$ are obtained from those in Subsection~\ref{subsec:AN} by replacing $\omega_2$ with $\omega_3$, and so we have
\begin{equation*}
  \Tor(\pattern;\rho)
  =
  \pm
  \frac{1}{2}.
\end{equation*}
Generators of the twisted homologies are given as follows:
\begin{itemize}
\item
$H_2(\pattern;\rho)=\C^2$ is generated by $\tilde{f}\otimes\tilde{U}$ and $\tilde{f}\otimes\tilde{V}$, and
\item
$H_1(\pattern;\rho)=\C^2$ is generated by $\tilde{p}\otimes\tilde{V}$ and $\tilde{t}\otimes\tilde{U}$,
\end{itemize}
where
$\tilde{U}:=\tilde{\Theta}_1\begin{pmatrix}2\\(\omega_3-\omega_3^{-1})z\\0\end{pmatrix}$
as before and
$\tilde{V}:=\tilde{\Theta}_1
  \begin{pmatrix}
    2 \\
    z+z^{-1}-2\omega_3^{-1}z \\
    2\omega_3^{-2}(\omega_3-1)(z^2-\omega_3)
  \end{pmatrix}
  =
  \begin{pmatrix}
    2 \\
    z-z^{-1} \\
    0
  \end{pmatrix}$,
which are obtained from $U$ and $V$ by replacing $\omega_2$ with $\omega_3$, respectively.
\par
Note that we have $\tilde{f}\otimes\tilde{U}=i^{\pattern}_{\ast}([S]\otimes\tilde{U})$ and $\tilde{f}\otimes\tilde{V}=i^{\pattern}_{\ast}([\Sigma]\otimes\tilde{V})$.
\subsubsection{Reidemeister torsion of the Mayer--Vietoris sequence}
From \eqref{eq:MV_torsion2} we have
\begin{equation}\label{eq:MV_torsion_NN}
  \Tor(E;\rho)
  =
  \frac{\Tor(C;\rho)\Tor(\pattern;\rho)}{\Tor(\mathcal{H}_{\ast})}
  =
  \pm
  \frac{2a+1}{4(\omega_1-\omega_1^{-1})^2\Tor(\mathcal{H}_{\ast})},
\end{equation}
where $\mathcal{H}_{\ast}$ is the Mayer--Vietoris exact sequence \eqref{eq:MV_exact}.
\par
We calculate the torsion of the Mayer--Vietoris exact sequence with respect to the bases that we have calculated:
\begin{align*}
  H_2(S;\rho)
  &=
  \left\langle
    [S]\otimes\tilde{U}
  \right\rangle,
  \\
  H_1(S;\rho)
  &=
  \left\langle
    \tilde{\mu}_{C}\otimes\tilde{U},\tilde{\lambda}_{C}\otimes\tilde{U}
  \right\rangle,
  \\
  H_0(S;\rho)
  &=
  \left\langle
    \tilde{v}\otimes\tilde{U}
  \right\rangle,
  \\
  H_2(C;\rho)
  &=
  \left\langle
    i^{C}_{\ast}([S]\otimes\tilde{U})
  \right\rangle,
  \\
  H_1(C;\rho)
  &=
  \left\langle
    \tilde{x}\otimes\tilde{U}
  \right\rangle,
  \\
  H_2(\pattern;\rho)
  &=
  \left\langle
    i^{\pattern}_{\ast}([S]\otimes\tilde{U}), i^{\pattern}_{\ast}([\Sigma]\otimes\tilde{V})
  \right\rangle,
  \\
  H_1(\pattern;\rho)
  &=
  \left\langle
    \tilde{p}\otimes\tilde{V}, \tilde{t}\otimes\tilde{U}
  \right\rangle.
\end{align*}
So we see that $H_0(E;\rho)=\{0\}$ and have the following exact sequence.
\begin{equation*}
\begin{CD}
  @>\delta_3>>\C
  @>\varphi_2>>\C\oplus\C^2
  @>\psi_2>>H_2(E;\rho)
  \\
  @>\delta_2>>\C^2
  @>\varphi_1>>\C\oplus\C^2
  @>\psi_1>>H_1(E;\rho)
  \\
  @>\delta_1>>\C
  @>>>\{0\}.
\end{CD}
\end{equation*}
\par
We will study the homomorphisms $\varphi_k$ ($k=1,2$).
\begin{itemize}
\item
$\varphi_2$:
As in the calculation in Subsection~\ref{subsec:AN}, $i^{\pattern}_{\ast}$ sends $[S]\otimes\tilde{U})$ to $i^{\pattern}_{\ast}([S]\otimes\tilde{U})$.
Moreover as in the calculation in Subsection~\ref{subsec:NA}, $i^{C}_{\ast}$ sends $[S]\otimes\tilde{U}$ to $i^{C}_{\ast}([S]\otimes\tilde{U})$.
Therefore $\varphi_2$ is injective.
\item
$\varphi_1$:
By calculation similar to that in Subsection~\ref{subsec:NA}, we see that $i^{C}_{\ast}$ sends $\tilde{\mu}_{C}\otimes\tilde{U}$ to $\tilde{x}\otimes\tilde{U}\in H_1(C;\rho)$ and $\tilde{\lambda}_{C}\otimes\tilde{U}$ to $-2(2a+1)\tilde{x}\otimes\tilde{U}\in H_1(C;\rho)$.
\par
By calculation similar to that in Subsection~\ref{subsec:AN}, we see that $i^{\pattern}_{\ast}$ sends $\tilde{\mu}_{C}\otimes\tilde{U}$ to $-2\tilde{t}\otimes\tilde{U}\in H_1(\pattern;\rho)$, and $\tilde{\lambda}_{C}\otimes\tilde{U}$ to $(2b+1)\tilde{t}\otimes\tilde{U}\in H_1(\pattern;\rho)$.
\par
So $\varphi_1$ is presented by the matrix
\begin{equation*}
  \begin{pmatrix}
     1&-2(2a+1) \\
     0&0  \\
    -2&2b+1
  \end{pmatrix}
\end{equation*}
with respect to the bases mentioned above.
It is clear that $\varphi_1$ is injective.
\end{itemize}
Therefore we have $H_0(E;\rho)=\{0\}$, $H_2(E;\rho)=\C^2$, which is generated by $\{(j^{\pattern}_{\ast}\circ i^{\pattern}_{\ast}([S]\otimes\tilde{U}),(j^{\pattern}_{\ast}\circ i^{\pattern}_{\ast})([\Sigma]\otimes\tilde{V})\}$, and $H_1(E;\rho)=\C^2$, which is generated by $\{j^{\pattern}_{\ast}(\tilde{p}\otimes\tilde{V}),\gamma\}$, where we choose $\gamma$ so that $\delta_1(\gamma)=\tilde{v}\otimes\tilde{U}$.
\par
Now the bases for $\mathcal{H}$ are given as follows:
\begin{align*}
  \mathbf{c}^{\mathcal{H}}_8
  &=
  \left\langle
    [S]\otimes U
  \right\rangle,
  \\
  \mathbf{c}^{\mathcal{H}}_7
  &=
  \left\langle
    i^{C}_{\ast}([S]\otimes\tilde{U}),
    i^{\pattern}_{\ast}([S]\otimes\tilde{U}),
    i^{\pattern}_{\ast}([\Sigma]\otimes\tilde{V})
  \right\rangle,
  \\
  \mathbf{c}^{\mathcal{H}}_6
  &=
  \left\langle
    (j^{\pattern}_{\ast}\ast i^{\pattern}_{\ast})([S]\otimes\tilde{U}),
    (j^{\pattern}_{\ast}\ast i^{\pattern}_{\ast})([\Sigma]\otimes\tilde{V})
  \right\rangle,
  \\
  \mathbf{c}^{\mathcal{H}}_5
  &=
  \left\langle
    \tilde{\mu}_{C}\otimes\tilde{U},
    \tilde{\lambda}_{C}\otimes\tilde{U}
  \right\rangle,
  \\
  \mathbf{c}^{\mathcal{H}}_4
  &=
  \left\langle
    \tilde{x}\otimes\tilde{U},
    \tilde{p}\otimes\tilde{V},
    \tilde{t}\otimes\tilde{U}
  \right\rangle,
  \\
  \mathbf{c}^{\mathcal{H}}_3
  &=
  \left\langle
    j^{\pattern}_{\ast}(\tilde{p}\otimes\tilde{V}),
    \gamma
  \right\rangle,
  \\
  \mathbf{c}^{\mathcal{H}}_2
  &=
  \left\langle
    \tilde{v}\otimes\tilde{U}
  \right\rangle.
\end{align*}
\par
We choose $\mathbf{b}^{\mathcal{H}}_i$ as follows:
\begin{align*}
  \\
  \mathbf{b}^{\mathcal{H}}_8
  &=
  \mathbf{c}^{\mathcal{H}}_8,
  \\
  \mathbf{b}^{\mathcal{H}}_7
  &=
  \left\langle
    i^{\pattern}_{\ast}([S]\otimes\tilde{U}),
    i^{\pattern}_{\ast}([\Sigma]\otimes\tilde{V})
  \right\rangle,
  \\
  \mathbf{b}^{\mathcal{H}}_5
  &=
  \mathbf{c}^{\mathcal{H}}_5,
  \\
  \mathbf{b}^{\mathcal{H}}_4
  &=
  \left\langle
    \tilde{p}\otimes\tilde{V}
  \right\rangle,
  \\
  \mathbf{b}^{\mathcal{H}}_3
  &=
  \left\langle
    \gamma
  \right\rangle.
\end{align*}
Then the torsion of the Mayer-Vietoris sequence equals
\begin{equation*}
\begin{split}
  \Tor(\mathcal{H}_{\ast})
  &=
  \pm
  \frac{
  \left[
    \partial(\mathbf{b}^{\mathcal{H}}_8)
    \cup
    \mathbf{b}^{\mathcal{H}}_7
    \mid
    \mathbf{c}^{\mathcal{H}}_7
  \right]
  \left[
    \mathbf{b}^{\mathcal{H}}_5
    \mid
    \mathbf{c}^{\mathcal{H}}_5
  \right]
  \left[
    \partial(\mathbf{b}^{\mathcal{H}}_4)
    \cup
    \mathbf{b}^{\mathcal{H}}_3
    \mid
    \mathbf{c}^{\mathcal{H}}_3
  \right]
  }
  {
  \left[
    \mathbf{b}^{\mathcal{H}}_8
    \mid
    \mathbf{c}^{\mathcal{H}}_8
  \right]
  \left[
    \partial(\mathbf{b}^{\mathcal{H}}_7)
    \mid
    \mathbf{c}^{\mathcal{H}}_6
  \right]
  \left[
    \partial(\mathbf{b}^{\mathcal{H}}_5)
    \cup
    \mathbf{b}^{\mathcal{H}}_4
    \mid
    \mathbf{c}^{\mathcal{H}}_4
  \right]
  \left[
    \partial(\mathbf{b}^{\mathcal{H}}_3)
    \mid
    \mathbf{c}^{\mathcal{H}}_2
  \right]
  }
  \\
  &=
  \pm
  \frac{1}
  {
  \left[
    \partial(\mathbf{b}^{\mathcal{H}}_5)
    \cup
    \mathbf{b}^{\mathcal{H}}_4
    \mid
    \mathbf{c}^{\mathcal{H}}_4
  \right]
  }
  \\
  &=
  \pm
  \frac{1}
  {
  \det
  \begin{pmatrix}
     1 &-2(2a+1)&0 \\
     0 &0       &1 \\
     -2&2b+1    &0
  \end{pmatrix}
  }
  \\
  &=
  \pm\frac{1}{(2b+1)-4(2a+1)}.
\end{split}
\end{equation*}
\par
So we finally have
\begin{equation*}
  \Tor(E;\rho)
  =
  \pm
  \frac{(2a+1)(4(2a+1)-(2b+1))}{4(\omega_1-\omega_1^{-1})^2}
\end{equation*}
from \eqref{eq:MV_torsion_NN}.
\begin{rem}
In \cite[3.6.3]{Murakami:ACTMV2014} the author calculated the twisted Alexander polynomial and tried in vain to obtain the twisted Reidemeister torsion.
\end{rem}

\bibliography{mrabbrev,hitoshi}
\bibliographystyle{amsplain}
\end{document}